\documentclass[10pt]{amsart}
      \usepackage[mathscr]{eucal}
      \usepackage{amsmath,amsfonts}
\def\rg{\hbox to 30pt{\rightarrowfill}}
\def\lg{\hbox to 30pt{\leftarrowfill}}

      \parskip\smallskipamount
          \newtheorem{theorem}{Theorem}[section]
      
      \newtheorem{proposition}[theorem]{Proposition}
      \newtheorem{corollary}[theorem]{Corollary}
      \newtheorem{lemma}[theorem]{Lemma}

      \newtheorem{remark}[theorem]{Remark}
      \makeatletter
      \@addtoreset{equation}{section}
      \makeatother

      \newcommand{\BB}{{\mathbb B}}
      \newcommand{\CC}{{\mathbb C}}
      \newcommand{\NN}{{\mathbb N}}
      
      \newcommand{\ZZ}{{\mathbb Z}}
      \newcommand{\DD}{{\mathbb D}}
      \newcommand{\RR}{{\mathbb R}}
      \newcommand{\FF}{{\mathbb F}}

      \newcommand{\cA}{{\mathcal A}}
      \newcommand{\cB}{{\mathcal B}}
      \newcommand{\cC}{{\mathcal C}}
      \newcommand{\cD}{{\mathcal D}}
      \newcommand{\cE}{{\mathcal E}}
      
      \newcommand{\cG}{{\mathcal G}}
      \newcommand{\cH}{{\mathcal H}}
      \newcommand{\cK}{{\mathcal K}}
      \newcommand{\cL}{{\mathcal L}}
      \newcommand{\cM}{{\mathcal M}}
      \newcommand{\cN}{{\mathcal N}}
      
      \newcommand{\cQ}{{\mathcal Q}}
      \newcommand{\cP}{{\mathcal P}}
      \newcommand{\cR}{{\mathcal R}}
      \newcommand{\cS}{{\mathcal S}}
      \newcommand{\cT}{{\mathcal T}}
      
      \newcommand{\cV}{{\mathcal V}}

      \newcommand{\supp}{\hbox{\rm{supp}}\,}
      \newcommand{\rank}{\hbox{\rm{rank}}\,}

      \newdimen\expt
      \def\boxit#1{\setbox0\hbox{$\displaystyle{#1}$}
            \hbox{\lower.4\expt
       \hbox{\lower3\expt\hbox{\lower\dp0
            \hbox{\vbox{\hrule height.4\expt
       \hbox{\vrule width.4\expt\hskip3\expt
            \vbox{\vskip3\expt\box0\vskip2\expt}%
       \hskip3\expt\vrule width.4\expt}\hrule height.4\expt}}}}}}
      
\begin{document}
       \pagestyle{myheadings}
      \markboth{ Gelu Popescu}{  Pluriharmonic functions on noncommutative polyballs  }

      \title [ Hyperbolic geometry on noncommutative  polyballs ]
      { Hyperbolic geometry on noncommutative  polyballs }
        \author{Gelu Popescu}
\date{October 30, 2016}
      \thanks{Research supported in part by  NSF grant DMS 1500922}

      \address{Department of Mathematics, The University of Texas
      at San Antonio \\ San Antonio, TX 78249, USA}
      \email{\tt gelu.popescu@utsa.edu}

\begin{abstract} This paper is an introduction to the hyperbolic geometry of noncommutative polyballs ${\bf B}_{\bf n}(\cH)$ in $B(\cH)^{n_1+\cdots +n_k}$, where  ${\bf n}=(n_1,\ldots, n_k)\in \NN^k$ and $B(\cH)$ is  the algebra
  of all bounded linear operators on a Hilbert space $\cH$.
We use the theory of free pluriharmonic functions on polyballs
   and noncommutative Poisson kernels on tensor products of full Fock spaces to define hyperbolic type metrics on ${\bf B}_{\bf n}(\cH)$, study their properties, and obtain hyperbolic versions of Schwarz-Pick lemma for free holomorphic functions on polyballs. As a consequence, the  polyballs can be viewed as noncommutative hyperbolic spaces. When specialized to the regular polydisk ${\bf D}^k(\cH)$ (which corresponds to the case $n_1=\cdots =n_k=1$), our hyperbolic metric $\delta_H$ is  complete  and invariant under the group $Aut({\bf D}^k)$ of all free holomorphic automorphisms  of ${\bf D}^k(\cH)$, and the $\delta_H$-topology induced on
${\bf D}^k(\cH)$ is the usual operator norm topology. The restriction of $\delta_H$ to the scalar polydisk $\DD^k$ is equivalent to the Kobayashi distance on $\DD^k$.
Most of the results of this paper are presented in the more general setting of Harnack (resp.~Poisson) parts of the closed polyball ${\bf B}_{\bf n}(\cH)^-$.

\end{abstract}

      \maketitle

\section*{Contents}
{\it

\quad Introduction

\begin{enumerate}
   \item[1.]     Harnack and Poisson equivalence relations  on the closed  polyball
\item[2.]  Hyperbolic metric on the Harnack parts of the closed  polyball
\item[3.] A metric   on the  Poisson parts of  the closed polyball
\item[4.] Hyperbolic metric on the regular  polydisk

\end{enumerate}

\quad References

}

\section*{Introduction}

A  theory of free holomorphic functions on noncommutative  polydomains which admit universal operator models has been  developed  in \cite{Po-holomorphic}, \cite{Po-pluriharmonic}, \cite{Po-automorphism}, \cite{Po-Berezin3}, \cite{Po-Berezin-poly}, and \cite{Po-automorphisms-polyball}.  These results played a crucial role in   our work on
  the curvature invariant \cite{Po-curvature-polyball}, the Euler characteristic  \cite{Po-Euler-charact}, and the  group of  free holomorphic automorphisms  on  noncommutative regular polyballs \cite{Po-automorphisms-polyball}. The regular  polyball ${\bf B_n}$, ${\bf n}=(n_1,\ldots, n_k)\in \NN^k$,  is a noncommutative analogue of the scalar polyball \, $(\CC^{n_1})_1\times\cdots \times  (\CC^{n_k})_1$   and  has a universal model $ {\bf S}:=\{{\bf S}_{i,j}\}$ consisting of   left creation
operators acting on the  tensor product $F^2(H_{n_1})\otimes \cdots \otimes F^2(H_{n_k})$ of full Fock spaces.  As a consequence, the theory of free holomorphic functions on noncommutative  polydomains   is related, via noncommutative Berezin transforms, to the study of the operator algebras generated by the universal models associated with the polydomains,  as well as to the theory of functions in several complex variable(\cite{Kr2}, \cite{Ru1}, \cite{Ru2}).  We remark that, in general, one can  view the free holomorphic functions on noncommutative polydomains  as  noncommutative functions in the sense of \cite{KV}.

 Recently \cite{Po-pluriharmonic-polyball}, we  obtained structure theorems characterizing   the  bounded (resp.~positive) free $k$-pluriharmonic functions on regular polyballs. These results will play an important role in the present paper which is  an introduction to the hyperbolic geometry of noncommutative polyballs. The main goal is to
  introduce  hyperbolic type
metrics  on these polyballs, study their basic  properties,  and provide an analogue of  Schwarz-Pick lemma in this setting. As a consequence, the regular polyballs can be viewed as noncommutative hyperbolic spaces.

Poincar\' e's  discovery of a conformally invariant metric on the
open unit disc $\DD:=\{z\in \CC: \ |z|<1\}$  of the complex plane
is at the heart of geometric  function theory. The
hyperbolic (Poincar\' e)  distance is defined on  $\DD$ by
$$
\delta_P(z,w):=\frac{1}{2}\ln
\frac{1+|\varphi_z(w)|}{1-|\varphi_z(w)|},\qquad
z,w\in \DD,
$$
where $\varphi_z$ is the automorphism of $\DD$ given by $\varphi_z(w)=\frac{w-z}{1-\bar z w}$,
 and  it  is invariant under the conformal
automorphisms of $\DD$, i.e.
$$
\delta_P(\varphi(z), \varphi(w))=\delta_P(z,w),\qquad z,w\in \DD,
$$
for all $\varphi\in Aut(\DD)$.
Moreover,
$(\DD, \delta_P)$ is a complete metric space and the $\delta_P$-topology induced on the open disk is the usual
planar topology.  Schwarz-Pick lemma asserts that
any analytic function $f:\DD\to \DD$ is distance-decreasing with respect to $\delta_P$,
i.e.
$$
\delta_P(f(z), f(w))\leq \delta_P(z,w),\qquad z,w\in \DD.
$$
This result has had profound implications in the development of geometric function theory. It has been generalized to higher dimensional complex spaces in various ways (see \cite{Ko1}, \cite{Ko2}).
Bergman (see \cite{Be}) introduced an analogue of the Poincar\' e
distance for the open unit ball
$\BB_m:=\{{\bf z} \in \CC^m:\ \|{\bf z}\|_2<1\}$,
   which has
 properties similar to those of $\delta_P$.  There is
a large literature concerning invariant metrics, hyperbolic
manifolds, and  the geometric viewpoint of complex function theory
(see \cite{Ko1}, \cite{Ko2}, \cite{JP}, \cite{Zhu}, and \cite{Kr1} and the
references there in).
There are  several extensions  of the   Poincar\' e-Bergman distance
and related topics to more general domains.
 We mention the work of  by L. Harris
(\cite{Ha1}, \cite{Ha2}, \cite{Ha3}) in the setting of
$JB^*$-algebras (see also the book by H.~Upmeier \cite{Up}), and the work of I.~Suciu
(\cite{Su2}, \cite{Su3}), Foia\c s (\cite{Fo}), and And\^
o-Suciu-Timotin (\cite{AST}) on Harnack parts   of contractions and
Harnack  type distances between two contractions on Hilbert spaces.
In \cite{Po-hyperbolic} (see also \cite{Po-hyperbolic3}, \cite{Po-hyperbolic2}),  we  introduced a hyperbolic
metric on the noncommutative ball
$$
[B(\cH)^m]_1:=\left\{(X_1,\ldots, X_m)\in B(\cH)^m:\ \|X_1
X_1^*+\cdots +X_mX_m^* \|^{1/2} <1\right\},\qquad m\in \NN,
$$
where $B(\cH)$ denotes  the algebra
  of all bounded linear operators on a Hilbert space $\cH$, which  is a noncommutative extension  of the Poincar\'
  e-Bergman metric on the open unit ball of $\CC^m$.
  We also  obtained a Schwarz-Pick lemma for free
holomorphic functions on $[B(\cH)^m]_1$   with respect to the
hyperbolic metric.

To present our  results, we introduce the  regular polyballs.
  We denote by  $B(\cH)^{n_1}\times_c\cdots \times_c B(\cH)^{n_k}$, where $n_i \in\NN:=\{1,2,\ldots\}$,
   the set of all tuples  ${\bf X}:=({ X}_1,\ldots, { X}_k)$ in $B(\cH)^{n_1}\times\cdots \times B(\cH)^{n_k}$
     with the property that the entries of ${X}_s:=(X_{s,1},\ldots, X_{s,n_s})$  are commuting with the entries of
      ${X}_t:=(X_{t,1},\ldots, X_{t,n_t})$  for any $s,t\in \{1,\ldots, k\}$, $s\neq t$.
  Note that  the operators $X_{s,1},\ldots, X_{s,n_s}$ are not necessarily commuting.
  Let ${\bf n}:=(n_1,\ldots, n_k)$ and define  the polyball
  $${\bf P_n}(\cH):=[B(\cH)^{n_1}]_1\times_c \cdots \times_c [B(\cH)^{n_k}]_1.
  $$
    If $A$ is  a positive invertible operator, we write $A>0$. The {\it regular polyball} on the Hilbert space $\cH$  is defined by
$$
{\bf B_n}(\cH):=\left\{ {\bf X}\in {\bf P_n}(\cH) : \ {\bf \Delta_{X}}(I)> 0  \right\},
$$
where
 the {\it defect mapping} ${\bf \Delta_{X}}:B(\cH)\to  B(\cH)$ is given by
$$
{\bf \Delta_{X}}:=\left(id -\Phi_{X_1}\right)\circ \cdots \circ\left(id -\Phi_{ X_k}\right),
$$
 and
$\Phi_{X_i}:B(\cH)\to B(\cH)$  is the completely positive linear map defined by
$$\Phi_{X_i}(Y):=\sum_{j=1}^{n_i}   X_{i,j} Y X_{i,j} ^*, \qquad Y\in B(\cH).
$$
 Note that if $k=1$, then ${\bf B_n}(\cH)$ coincides with the noncommutative unit ball $[B(\cH)^{n_1}]_1$.
   We remark that the scalar representation of
  the  ({\it abstract})
 {\it regular polyball} ${\bf B}_{\bf n}:=\{{\bf B_n}(\cH):\ \cH \text{\ is a Hilbert space} \}$ is
   ${\bf B_n}(\CC)={\bf P_n}(\CC)= (\CC^{n_1})_1\times \cdots \times (\CC^{n_k})_1$.
Extending Poincar\' e's result \cite{Kr2} that the open ball of $\CC^m$ is not biholomorphic equivalent to the polydisk $\DD^k$, $k\geq 2$, we proved in \cite{Po-automorphisms-polyball} that the noncommutative ball $[B(\cH)^m]_1$ is not free biholomorphic equivalent to the regular polyball ${\bf B_n}(\cH)$, where ${\bf n}=(n_1,\ldots, n_k)$ and $k\geq 2$. As in the classical case (\cite{Ru1}, \cite{Ru2}, \cite{Kr2}), one expects significant differences between the hyperbolic geometry of the ball $[B(\cH)^m]_1$ and that of the regular polyball ${\bf B_n}(\cH)$, and  differences  regarding the theory of free holomorphic (resp. pluriharmonic) functions on these noncommutative domains.

If ${\bf A},{\bf B}\in {\bf B_n}(\cH)^-$, we say that ${\bf A}$ and ${\bf B}$  are {\it Harnack
equivalent} (and denote ${\bf A}\overset{H}{\sim}\, {\bf B}$)  if
there exists $c>  1$ such that
\begin{equation*}
\frac{1}{c^2}F(r{\bf B})\leq  F(r{\bf A})\leq c^2 F(r{\bf B}), \qquad r\in [0,1),
\end{equation*}
 for any positive
free $k$-pluriharmonic function $F:{\bf B}_{\bf n}(\cH)\to B(\cE)\otimes_{min} B(\cH)$, where $\cE$ is a separable Hilbert space, in the sense of \cite{Po-pluriharmonic-polyball}. In this case, we write ${\bf A}\overset{H}{{\underset{c}\sim}}\, {\bf B}$.  The equivalence classes with respect to the equivalence relation
$\overset{H}\sim$ are called Harnack parts of ${\bf B_n}(\cH)^-$.
In Section 1, we prove a Harnack type inequality for
 positive
free $k$-pluriharmonic functions on regular polyballs and use it  to show that the Harnack part containing the zero element coincides with the open polyball ${\bf B_n}(\cH)$.

Given a Harnack part $\Delta$ of ${\bf B_n}(\cH)^-$, we define the map
$\delta_H:\Delta\times \Delta\to \RR^+$ by setting
$$
\delta_H({\bf A}, {\bf B}):=\ln \inf\left\{ c>1:\ {\bf A}\overset{H}{{\underset{c}\sim}}\, {\bf B}\right\}.
$$
In Section 2, we prove that $\delta_H$ is a metric on $\Delta$ and provide a Schwarz-Pick type result for free holomorphic functions on regular polybals with respect to $\delta_H$. We show  that if
$\Phi=(\Phi_1,\ldots, \Phi_m):{\bf B}_{\bf n} (\cH)\to   [B(\cH)^m]_1^-$  is a  free holomorphic function on the regular polyball and
 ${\bf X}, {\bf Y}\in {\bf B_n}(\cH)$,
then $\Phi({\bf X})\,\overset{H}{{ \sim}}\, \Phi({\bf Y})$ and
$$
\delta_H(\Phi({\bf X}), \Phi({\bf Y}))\leq \delta_H({\bf X}, {\bf Y}),
$$
where $\delta_H$ is the hyperbolic  metric  defined on  the Harnack parts of $[B(\cH)^m]_1^-$ and  on the polyball ${\bf B_n}(\cH)$, respectively.
Using the description of the group $Aut({\bf B}_{\bf n})$ of all free holomorphic automorphisms of ${\bf B}_{\bf n}$ (see \cite{Po-automorphisms-polyball}), we prove that
 $$
\delta_H({\bf A}, {\bf B})=\delta_H(\boldsymbol\Psi({\bf A}), \boldsymbol\Psi({\bf B})), \qquad \boldsymbol\Psi\in
Aut({\bf B_n}),
$$
for any
 ${\bf A}, {\bf B} \in {\bf B_n}(\cH)^-$ such that ${\bf A}\overset{H}{\sim}\, {\bf B}$.
In particular, the hyperbolic distance $\delta_H$ on the open polyball is invariant under the automorphism group $Aut({\bf B_n})$.

If ${\bf A}$ and ${\bf B}$  are  in
${\bf B_n}(\cH)^-$, we say that they are {\it Poisson
equivalent} (and denote ${\bf A}\overset{P}{\sim}\, {\bf B}$)  if
there exists $c>  1$ such that
$$\frac{1}{c^2}\boldsymbol{\cP}({\bf R}, r{\bf B})\leq  \boldsymbol{\cP}({\bf R}, r{\bf A})\leq c^2 \boldsymbol{\cP}({\bf R}, r{\bf B}), \qquad r\in [0,1), $$
where ${\bf X}\mapsto \boldsymbol{\cP}({\bf R}, {\bf X})$ is the {\it free pluriharmonic
Poisson kernel} on the regular polyball (see Section 1).
   In this case we write ${\bf A}\overset{P}{{\underset{c}\sim}}\, {\bf B}$.  The equivalence classes with respect to equivalence relation
$\overset{H}\sim$ are called Poisson parts of ${\bf B_n}(\cH)^-$.
We prove in Section 1 that  the Poisson part containing the zero element coincides with the open polyball ${\bf B_n}(\cH)$.

Given a Poisson part $\Delta$ of ${\bf B_n}(\cH)^-$, we define the map
$\delta_\cP:\Delta\times \Delta\to \RR^+$ by setting
$$
\delta_\cP({\bf A}, {\bf B}):=\ln \inf\left\{ c>1:\ {\bf A}\overset{P}{{\underset{c}\sim}}\, {\bf B}\right\}.
$$
In Section 3, we prove that $\delta_\cP$ is a metric on $\Delta$ and obtain an explicit  formula for it (see Theorem \ref{dp})  in terms of certain noncommutative Cauchy kernels acting on tensor products of full Fock spaces.
Moreover, we prove that $\delta_\cP$ is a complete metric on ${\bf B_n}(\cH)$ and that the $\delta_\cP$-topology coincides with the operator norm topology on ${\bf B_n}(\cH)$.

In Section 4, we consider the regular polydisk ${\bf D}^k(\cH):={\bf B}_{(1,\ldots, 1)}(\cH)$, which consists of all tuples ${\bf X}=(X_1,\ldots, X_k)$ of commuting strict contractions such that ${\bf \Delta_{X}}(I)>0$. We remark that in this case we have
$$
{\bf \Delta_{X}}(I)=\sum_{p_1,\ldots, p_k\in \{0,1\}} (-1)^{p_1+\cdots +p_k} T_1^{p_1}\cdots T_k^{p_k} (T_k^*)^{p_k}\cdots( T_1^*)^{p_1}
$$
and  ${\bf D}^k(\CC)=\DD^k$. Using a characterization of positive free $k$-pluriharmonic functions on regular polydisks (see Theorem \ref{cp}) and the results of the previous sections, we prove that    $ {\bf A}\overset{H}{{\underset{c}\sim}}\, {\bf B}$ if and only if $ {\bf A}\overset{P}{{\underset{c}\sim}}\, {\bf B}$,  for any
 ${\bf A}, {\bf B} \in {\bf D}^k(\cH)^-$. Consequently, the metrics $\delta_H$ and $\delta_\cP$ coincide on the Harnack (resp.~Poisson) parts of $ {\bf D}^k(\cH)^-$.   We show that the hyperbolic metric $\delta_H$ is complete on
 the regular polydisk and the $\delta_H$-topology coincides with the operator norm topology on ${\bf D}^k(\cH)$. As a consequence, $({\bf D}^k(\cH), \delta_H)$ is a complete hyperbolic space. Moreover,  $\delta_H$ is invariant under the automorphism group $Aut({\bf D}^k)$ and
  $$
\delta_H(f({\bf X}), f({\bf Y}))\leq \delta_H({\bf X}, {\bf Y}),\qquad {\bf X}, {\bf Y}\in {\bf D}^k(\cH)
$$
for any   free holomorphic function $f:{\bf D}^k(\cH)\to [B(\cH)^m]_1$.
Therefore, the hyperbolic metric $\delta_H$ on ${\bf D}^k(\cH)$ has similar properties  to those of the Poincar\' e distance on the open unit disc $\DD$.
 In addition, we prove that if
  ${\bf A}$ and ${\bf B}$ are   in
 ${\bf D}^k(\cH)$, then
$$
\delta_H({\bf A}, {\bf B})=\ln \max \left\{\left\|C_{\bf A}({\bf R})C_{\bf B}({\bf R})^{-1}\right\|, \left\|C_{\bf B}({\bf R})C_{\bf A}({\bf R})^{-1}\right\|\right\},
$$
where
$$C_{\bf X}({\bf R}):=(I\otimes \boldsymbol\Delta_{\bf X}(I)^{1/2})\prod_{i=1}^k(I-R_{i}\otimes
X_{i}^*), \qquad {\bf X}=(X_1,\ldots, X_k)\in {\bf D}^k(\cH), $$
and $R_1,\ldots R_k$ are the shift operators on the Hardy space $H^2(\DD^k)$.
 In particular, we show that
$\delta_H|_{\DD^k\times \DD^k}$ is equivalent to the  Kobayashi distance on the polydisk $\DD^k$  (see \cite{JP}) and
$$\delta_H({\bf z}, {\bf w})=\frac{1}{2} \ln \frac{\prod_{i=1}^k\left(1+|\psi_{z_i}(w_i)|\right)}
{\prod_{i=1}^k\left(1-|\psi_{z_i}(w_i)|\right)}
$$
for any ${\bf z}=(z_1,\ldots, z_k)$ and ${\bf w}=(w_1,\ldots, w_k)$ in $\DD^k$, where $\psi_{\bf z}:=(\psi_{z_1},\ldots, \psi_{z_n})$ is the involutive automorphisms of $\DD^k$ such that
$\psi_{z_i}(0)=z_i$ and $\psi_{z_i}(z_i)=0$.

We remark that, according to  the results of the present paper and  those from \cite{Po-hyperbolic}, the metrics $\delta_H$ and $\delta_\cP$ coincide   for the regular polydisk ${\bf D}^k(\cH)$ and  for  the noncommutative  ball $[B(\cH)^{n_1}]_1$. It remains an open problem whether the same result is true for any regular polyball.
On the other hand, it will be interesting to see to what extent these results extend to more general polydomains in $B(\cH)^m$ such as those studied in  \cite{Po-Berezin3} and \cite{Po-Berezin-poly}.

\bigskip
\section{Harnack and Poisson  equivalences on the closed  polyball}

In this section, we recall some basic facts concerning the noncommutative Berezin transforms on polyballs and introduce the Harnack  and the  Poisson equivalence relations on the closed polyball ${\bf B}_{\bf n}(\cH)^-$. We provide a Harnack type inequality for positive free $k$-pluriharmonic functions  and use it to show that the Harnack (resp.~Poisson) equivalence class containing the zero element  coincides with the open polyball  ${\bf B}_{\bf n}(\cH)$.

Let $H_{n_i}$ be
an $n_i$-dimensional complex  Hilbert space with orthonormal basis $e^i_1,\ldots, e^i_{n_i}$.
  We consider the {\it full Fock space}  of $H_{n_i}$ defined by
$F^2(H_{n_i}):=\CC 1 \oplus\bigoplus_{p\geq 1} H_{n_i}^{\otimes p},$
where  $H_{n_i}^{\otimes p}$ is the
(Hilbert) tensor product of $p$ copies of $H_{n_i}$. Let $\FF_{n_i}^+$ be the unital free semigroup on $n_i$ generators
$g_{1}^i,\ldots, g_{n_i}^i$ and the identity $g_{0}^i$.
  Set $e_\alpha^i :=
e^i_{j_1}\otimes \cdots \otimes e^i_{j_p}$ if
$\alpha=g^i_{j_1}\cdots g^i_{j_p}\in \FF_{n_i}^+$
 and $e^i_{g^i_0}:= 1\in \CC$.
  The length of $\alpha\in
\FF_{n_i}^+$ is defined by $|\alpha|:=0$ if $\alpha=g_0^i$  and
$|\alpha|:=p$ if
 $\alpha=g_{j_1}^i\cdots g_{j_p}^i$, where $j_1,\ldots, j_p\in \{1,\ldots, n_i\}$.
 We  define
 the {\it left creation  operator} $S_{i,j}$ acting on the  Fock space $F^2(H_{n_i})$  by setting
$
S_{i,j} e_\alpha^i:=  e^i_j\otimes e^i_{ \alpha}$, $\alpha\in \FF_{n_i}^+,
$
 and
  the operator  ${\bf S}_{i,j}$ acting on the Hilbert tensor  product
$F^2(H_{n_1})\otimes\cdots\otimes F^2(H_{n_k})$ by setting
$${\bf S}_{i,j}:=\underbrace{I\otimes\cdots\otimes I}_{\text{${i-1}$
times}}\otimes \,S_{i,j}\otimes \underbrace{I\otimes\cdots\otimes
I}_{\text{${k-i}$ times}},
$$
where  $i\in\{1,\ldots,k\}$ and  $j\in\{1,\ldots,n_i\}$.  We denote ${\bf S}:=({\bf S}_1,\ldots, {\bf S}_k)$, where  ${\bf S}_i:=({\bf S}_{i,1},\ldots,{\bf S}_{i,n_i})$, or ${\bf S}:=\{{\bf S}_{i,j}\}$.   The noncommutative Hardy algebra ${\bf F}_{\bf n}^\infty$ (resp.~the polyball algebra $\boldsymbol\cA_{\bf n}$) is the weakly closed (resp.~norm closed) non-selfadjoint  algebra generated by $\{{\bf S}_{i,j}\}$ and the identity.
Similarly, we   define
 the {\it right creation  operator} $R_{i,j}:F^2(H_{n_i})\to
F^2(H_{n_i})$     by setting
 $
R_{i,j} e_\alpha^i:=  e^i_ {\alpha }\otimes e^i_j$ for $ \alpha\in \FF_{n_i}^+,
$
 and
  the operator  ${\bf R}_{i,j}$ acting on the Hilbert tensor  product
$F^2(H_{n_1})\otimes\cdots\otimes F^2(H_{n_k})$ by setting
$${\bf R}_{i,j}:=\underbrace{I\otimes\cdots\otimes I}_{\text{${i-1}$
times}}\otimes \,R_{i,j}\otimes \underbrace{I\otimes\cdots\otimes
I}_{\text{${k-i}$ times}}.
$$
  The polyball algebra $\boldsymbol\cR_{\bf n}$ is the norm closed non-selfadjoint  algebra generated by $\{{\bf R}_{i,j}\}$ and the identity.

We recall (see \cite{Po-poisson}, \cite{Po-Berezin-poly}) some basic properties for  the   noncommutative Berezin     transforms associated  with regular polyballs.
 Let  ${\bf X}=({ X}_1,\ldots, { X}_k)\in {\bf B_n}(\cH)^-$ with $X_i:=(X_{i,1},\ldots, X_{i,n_i})$.
We  use the notation
$X_{i,\alpha_i}:=X_{i,j_1}\cdots X_{i,j_p}$
  if  $\alpha_i=g_{j_1}^i\cdots g_{j_p}^i\in \FF_{n_i}^+$ and
   $X_{i,g_0^i}:=I$.
The {\it noncommutative Berezin kernel} associated with any element
   ${\bf X}$ in the noncommutative polyball ${\bf B_n}(\cH)^-$ is the operator
   $${\bf K_{X}}: \cH \to F^2(H_{n_1})\otimes \cdots \otimes  F^2(H_{n_k}) \otimes  \overline{{\bf \Delta_{X}}(I) (\cH)}$$
   defined by
   $$
   {\bf K_{X}}h:=\sum_{\beta_i\in \FF_{n_i}^+, i=1,\ldots,k}
   e^1_{\beta_1}\otimes \cdots \otimes  e^k_{\beta_k}\otimes {\bf \Delta_{X}}(I)^{1/2} X_{1,\beta_1}^*\cdots X_{k,\beta_k}^*h, \qquad h\in \cH,
   $$
   where $ {\bf \Delta_{X}}(I)$ is the defect operator.
A very  important property of the Berezin kernel is that
     ${\bf K_{X}} { X}^*_{i,j}= ({\bf S}_{i,j}^*\otimes I)  {\bf K_{X}}$ for any  $i\in \{1,\ldots, k\}$ and $ j\in \{1,\ldots, n_i\}.
    $
    The {\it Berezin transform at} ${\bf X}\in {\bf B_n}(\cH)$ is the map $ \boldsymbol{\cB_{\bf X}}: B(\otimes_{i=1}^k F^2(H_{n_i}))\to B(\cH)$
 defined by
\begin{equation*}
 {\boldsymbol\cB_{\bf X}}[g]:= {\bf K^*_{\bf X}} (g\otimes I_\cH) {\bf K_{\bf X}},
 \qquad g\in B(\otimes_{i=1}^k F^2(H_{n_i})).
 \end{equation*}
  If $g$ is in   the $C^*$-algebra $C^*({\bf S})$ generated by ${\bf S}_{i,1},\ldots,{\bf S}_{i,n_i}$, where $i\in \{1,\ldots, k\}$, we  define the Berezin transform at  $X\in {\bf B_n}(\cH)^-$  by
  $${\boldsymbol\cB_{\bf X}}[g]:=\lim_{r\to 1} {\bf K^*_{r\bf X}} (g\otimes I_\cH) {\bf K_{r\bf X}},
 \qquad g\in  C^*({\bf S}),
 $$
 where the limit is in the operator norm topology.
In this case, the Berezin transform at $X$ is a unital  completely positive linear  map such that
 $${\boldsymbol\cB_{\bf X}}({\bf S}_{\boldsymbol\alpha} {\bf S}_{\boldsymbol\beta}^*)={\bf X}_{\boldsymbol\alpha} {\bf X}_{\boldsymbol\beta}^*, \qquad \boldsymbol\alpha, \boldsymbol\beta \in \FF_{n_1}^+\times \cdots \times\FF_{n_k}^+,
 $$
 where  ${\bf S}_{\boldsymbol\alpha}:= {\bf S}_{1,\alpha_1}\cdots {\bf S}_{k,\alpha_k}$ if  $\boldsymbol\alpha:=(\alpha_1,\ldots, \alpha_k)\in \FF_{n_1}^+\times \cdots \times\FF_{n_k}^+$.

The  Berezin transforms will play an important role in this paper.
   More properties  concerning  noncommutative Berezin transforms and multivariable operator theory on noncommutative balls and  polydomains, can be found in \cite{Po-poisson}      and \cite{Po-Berezin-poly}.
For basic results on completely positive (resp.~bounded)  maps  we refer the reader to \cite{Pa-book} and \cite{Pi-book} .

For each $m\in \ZZ$, we set $m^+:=\max\{m, 0\}$ and $m^-:=\max\{-m, 0\}$.
 A function $F$ with operator-valued coefficients in $B(\cE)$, where $\cE$ is separable Hilbert space, is called {\it free $k$-pluriharmonic} on the abstract polyball ${\bf B_n}$    if it has the form
$$
F({\bf X})= \sum_{m_1\in \ZZ}\cdots \sum_{m_k\in \ZZ} \sum_{{\alpha_i,\beta_i\in \FF_{n_i}^+, i\in \{1,\ldots, k\}}\atop{|\alpha_i|=m_i^-, |\beta_i|=m_i^+}}A_{(\alpha_1,\ldots,\alpha_k;\beta_1,\ldots, \beta_k)}
\otimes {\bf X}_{1,\alpha_1}\cdots {\bf X}_{k,\alpha_k}{\bf X}_{1,\beta_1}^*\cdots {\bf X}_{k,\beta_k}^*,
$$
where the multi-series converge in the operator norm topology for any  ${\bf X}=(X_1,\ldots, X_k)\in {\bf B_n}(\cH)$, with $X_i:=(X_{i,1},\ldots, X_{i,n_i})$, and any Hilbert space $\cH$. Without loss of generality, we can assume throughout this paper that $\cH$ is a separable infinite dimensional  Hilbert space.
According to \cite{Po-pluriharmonic-polyball}, the order of the series in the definition above is irrelevant. Note that any free holomorphic function on ${\bf B_n}$ is $k$-pluriharmonic. Indeed, according to \cite{Po-automorphisms-polyball}, any free holomorphic function on the polyball ${\bf B_n}$  with coefficients in $B(\cE)$ has the form
$$
f({\bf X})= \sum_{m_1\in \NN}\cdots \sum_{m_k\in \NN} \sum_{{\alpha_i\in \FF_{n_i}^+, i\in \{1,\ldots, k\}}\atop{|\alpha_i|=m_i}}C_{(\alpha_1,\ldots,\alpha_k)}
\otimes {\bf X}_{1,\alpha_1}\cdots {\bf X}_{k,\alpha_k},\qquad {\bf X}\in {\bf B_n}(\cH),
$$
where the multi-series converge in the operator norm topology.

Now, we introduce  a preorder relation
$\overset{H}{\prec} $ on the closed ball ${\bf B_n}(\cH)^-$.
If ${\bf A}$ and ${\bf B}$ are in
${\bf B_n}(\cH)^-$, we say that ${\bf A}$ is {\it Harnack dominated} by ${\bf B}$, and
denote ${\bf A}\overset{H}{\prec}\, {\bf B}$, if there exists $c>0$ such that
$$F(r{\bf A})\leq c^2 F(r{\bf B}) $$
 for any positive
free $k$-pluriharmonic function $F$  with operator valued coefficients
and any $r\in [0,1)$.
 When we
want to emphasize the constant $c$, we write
${\bf A}\overset{H}{{\underset{c}\prec}}\, {\bf B}$.
Since $\overset{H}{\prec} $ is a preorder
relation on ${\bf B_n}(\cH)^-$, it induces an equivalence relation
$\overset{H}\sim$ on ${\bf B_n}(\cH)^-$, which we call {\it Harnack
equivalence}. The equivalence classes with respect to
$\overset{H}\sim$ are called {\it Harnack parts} of ${\bf B_n}(\cH)^-$. Let
 ${\bf A}$ and ${\bf B}$ are in
${\bf B_n}(\cH)^-$. It is easy to see that ${\bf A}$ and ${\bf B}$ are Harnack
equivalent (we denote ${\bf A}\overset{H}{\sim}\, {\bf B}$)  if and only if
there exists $c\geq  1$ such that

\begin{equation}
\label{def-H}
\frac{1}{c^2}F(r{\bf B})\leq  F(r{\bf A})\leq c^2 F(r{\bf B})
\end{equation}
 for any positive
free $k$-pluriharmonic function $F$  with operator-valued coefficients
and any $r\in [0,1)$. We also use the
notation ${\bf A}\overset{H}{{\underset{c}\sim}}\, {\bf B}$ if
${\bf A}\overset{H}{{\underset{c}\prec}}\, {\bf B}$ and
${\bf B}\overset{H}{{\underset{c}\prec}}\, {\bf A}$.

We denote by $C^*({\bf S})$ the $C^*$-algebra generated by ${\bf S}_{i,j}$, where $i\in \{1,\ldots, k\}$ and $j\in \{1,\ldots, n_i\}$.
 A completely positive  (c.p.) linear map  $\mu_{\bf X}:C^*({\bf S})\to
 B(\cH)$ is called  {\it representing c.p. map} for the point
 ${\bf X}\in {\bf B_n}(\cH)^-$ if
 $$
 \mu_{\bf X}({\bf S}_{1,\alpha_1}\cdots {\bf S}_{k,\alpha_k}{\bf S}_{1,\beta_1}^*\cdots {\bf S}_{k,\beta_k}^*)={\bf X}_{1,\alpha_1}\cdots {\bf X}_{k,\alpha_k}{\bf X}_{1,\beta_1}^*\cdots {\bf X}_{k,\beta_k}^*
 $$
for any
$  (\alpha_1,\ldots, \alpha_k)$ and $(\beta_1,\ldots, \beta_k)$ in
 $\FF_{n_1}^+\times \cdots \times \FF_{n_k}^+$ with  $\alpha_i,\beta_i\in \FF_{n_i}^+$, $|\alpha_i|=m_i^-, |\beta_i|=m_i^+$, and $m_i\in \ZZ$.

  Next, we obtain   characterizations for the Harnack equivalence on the closed regular polyball
${\bf B_n}(\cH)^-$.

\begin{proposition}
\label{equivalent2} Let
 ${\bf A}$ and ${\bf B}$ are in
${\bf B_n}(\cH)^-$ and let $c>1$. Then the following
statements are equivalent:
\begin{enumerate}
\item[(i)]
${\bf A}\overset{H}{{\underset{c}\sim}}\, {\bf B}$;

\item[(ii)]
$  c^2\boldsymbol\cB_{\bf B}-\boldsymbol\cB_{\bf A}$  and  $  c^2\boldsymbol\cB_{\bf A}-\boldsymbol\cB_{\bf B}$ are completely positive
linear map on the operator space $\text{\rm span}\{ \boldsymbol\cA_{\bf n}^* \boldsymbol\cA_{\bf n}\}^{-\|\cdot\|}$, where $\boldsymbol\cB_{\bf X}$ is the noncommutative Berezin
transform at ${\bf X}\in {\bf B_n}(\cH)^-$;
\item[(iii)] there are   representing c.p. maps
$\mu_{\bf A}$ and $\mu_{\bf B}$ for ${\bf A}$ and ${\bf B}$, respectively, such that
$$
\frac{1}{c^2} \mu_{\bf B}\leq \mu_{\bf A}\leq c^2 \mu_{\bf B}.
$$
\end{enumerate}
\end{proposition}
\begin{proof}
 Let ${\bf A}$ and ${\bf B}$ be elements in the regular closed polyball
${\bf B_n}(\cH)^-$ and let $c>1$. First we prove that  ${\bf A}\overset{H}{{\underset{c}\prec}}\, {\bf B}$ if and only if
$  c^2\boldsymbol\cB_{\bf B}-\boldsymbol\cB_{\bf A}$   is a completely positive linear map on the
operator space $\text{\rm span}\{ \boldsymbol\cA_{\bf n}^* \boldsymbol\cA_{\bf n}\}^{-\|\cdot\|}$. Assume that  ${\bf A}\overset{H}{{\underset{c}\prec}}\, {\bf B}$ and let $g\in \boldsymbol\cP:=  \text{\rm span}\, \{\boldsymbol\cA_{\bf n}^*\boldsymbol\cA_{\bf n}\}^{-\|\cdot \|}$ be a positive operator. According to Theorem 2.4 from \cite{Po-pluriharmonic-polyball}, the map
$$
F({\bf X})=\boldsymbol{\cB}_{\bf X} [g]:=  {\bf K}^*_{\bf X}[g\otimes I_\cH] {\bf K}_{\bf X} , \qquad {\bf X}\in {\bf B}_{\bf n}(\cH),
  $$
 is a positive free $k$-pluriharmonic function on ${\bf B}_{\bf n}(\cH)$ which
 has a continuous extension  (in the operator norm topology) to
the closed ball ${\bf B}_{\bf n}(\cH)^-$. Since ${\bf A}\overset{H}{{\underset{c}\prec}}\, {\bf B}$, we have
$F(r{\bf A})\leq c^2F(r{\bf B})$ for any $r\in [0,1)$, which is equivalent to
$$
\left(c^2\boldsymbol{\cB}_{{
r\bf B}}-\boldsymbol{\cB}_{{
r\bf A}}\right)[g]\geq 0, \qquad r\in [0,1).
$$
Since $\boldsymbol{\cB}_{{
\bf B}}[g]:=\lim_{r\to 1} \boldsymbol{\cB}_{{
r\bf B}}[g]=F({\bf B})$ exists in the operator norm topology, and a similar result holds if we replace ${\bf B}$ with ${\bf A}$, we deduce that
$c^2\boldsymbol{\cB}_{{
\bf B}}-\boldsymbol{\cB}_{{
\bf A}}$ is a positive linear map on  $\text{\rm span}\{ \boldsymbol\cA_{\bf n}^* \boldsymbol\cA_{\bf n}\}^{-\|\cdot\|}$. Similarly, passing to matrices one ca prove that $c^2\boldsymbol{\cB}_{{
\bf B}}-\boldsymbol{\cB}_{{
\bf A}}$ is completely positive. Hence,  we deduce that  $ c^2\boldsymbol\cB_{\bf B}-\boldsymbol\cB_{\bf A}$ is completely positive  on the operator space $\text{\rm span}\{ \boldsymbol\cA_{\bf n}^* \boldsymbol\cA_{\bf n}\}^{-\|\cdot\|}$.

Conversely, assume that  $ c^2\boldsymbol\cB_{\bf B}-\boldsymbol\cB_{\bf A}$ is completely positive  on the operator space $\text{\rm span}\{ \boldsymbol\cA_{\bf n}^* \boldsymbol\cA_{\bf n}\}^{-\|\cdot\|}$. Then $c^2\boldsymbol{\cB}^{ext}_{{
\bf B}}-\boldsymbol{\cB}^{ext}_{{
\bf A}}$ is   positive on  $\boldsymbol\cP_\cE:=B(\cE)\otimes_{min} \text{\rm span}\, \{\boldsymbol\cA_{\bf n}^*\boldsymbol\cA_{\bf n}\}^{-\|\cdot \|}$, where
$$
\boldsymbol{\cB}_{\bf X}^{ext}[g]:=(I_\cE\otimes {\bf K}^*_{\bf X})[g\otimes I_\cH](I_\cE\otimes  {\bf K}_{\bf X}), \qquad {\bf X}\in {\bf B}_{\bf n}(\cH).
$$
 Let $F:{\bf B}_{\bf n}(\cH)\to B(\cE)\otimes_{min} B(\cH)$ be a positive free $k$-pluriharmonic function on ${\bf B}_{\bf n}(\cH)$. Then, for each $r\in [0,1)$, we have
 $$
 F(r{\bf X})=\boldsymbol{\cB}^{ext}_{\bf X}[F(r{\bf S})]\geq 0
 $$
 and $F(r{\bf S})\in \boldsymbol\cP_\cE$. Consequently, we have
 $$
 c^2F(r{\bf B})-F(r{\bf A})=\left(c^2\boldsymbol{\cB}^{ext}_{{
\bf B}}-\boldsymbol{\cB}^{ext}_{{
\bf A}}\right)[F(r{\bf S})]\geq 0
 $$
 for any $r\in [0,1)$, which proves our assertion. Now, the equivalence of (i) and (ii) is  clear.

It remains to prove that (ii) is equivalent to (iii). To this end,  assume that item (ii) holds. According to Arveson's extension theorem \cite{Arv-acta}, there are completely positive maps $\varphi, \psi:C^*({\bf S})\to B(\cH)$ such that
$$
\varphi(g)=c^2\boldsymbol\cB_{\bf A}[g]-\boldsymbol\cB_{\bf B}[g]\quad \text{ and } \quad \psi(g)=c^2\boldsymbol\cB_{\bf B}[g]-\boldsymbol\cB_{\bf A}[g]
$$
for any $g\in \text{\rm span}\{ \boldsymbol\cA_{\bf n}^* \boldsymbol\cA_{\bf n}\}^{-\|\cdot\|}$. Hence, we deduce that
$$
\boldsymbol\cB_{\bf A}[g]=\frac{c^2}{c^4-1}(c^2\varphi(g)+\psi(g))\quad \text{and } \quad \boldsymbol\cB_{\bf B}[g]=\frac{c^2}{c^4-1}(c^2\psi(g)+\varphi(g)).
$$
 Now, we define $\mu_{\bf A}:C^*({\bf S})\to B(\cH)$ and $\mu_{\bf B}:C^*({\bf S})\to B(\cH)$ by setting
 \begin{equation}\label{muab}
 \mu_{\bf A}:=\frac{c^2}{c^4-1}(c^2\varphi+\psi)\quad \text{ and }\quad
 \mu_{\bf B}:=\frac{c^2}{c^4-1}(c^2\psi+\varphi)
 \end{equation}
 and note that $\boldsymbol\cB_{\bf A}[g]=\mu_{\bf A}(g)$ and
 $\boldsymbol\cB_{\bf B}[g]=\mu_{\bf B}(g)$ for any $g\in \text{\rm span}\{ \boldsymbol\cA_{\bf n}^* \boldsymbol\cA_{\bf n}\}^{-\|\cdot\|}$.
 Due to the properties of the noncommutative  Berezin transform, it is clear that 
$\mu_{\bf A}$ and $\mu_{\bf B}$ are representing c.p. maps for ${\bf A}$ and ${\bf B}$, respectively. The inequalities $
\frac{1}{c^2} \mu_{\bf B}\leq \mu_{\bf A}\leq c^2 \mu_{\bf B}
$
are simple consequences of  relation \eqref{muab} and the fact that $c>1$.
To complete the proof, it is enough to  prove that (iii)$\implies$(i). To this end, assume that (iii) holds for some $c>1$ and let $F:{\bf B}_{\bf n}(\cH)\to B(\cE)\otimes_{min} B(\cH)$ be a positive free $k$-pluriharmonic function on ${\bf B}_{\bf n}(\cH)$. Then $F(r{\bf S})\geq 0$ for any $r\in [0,1)$.
Since $F(r{\bf S})\in B(\cE)\otimes_{min} \text{\rm span}\, \{\boldsymbol\cA_{\bf n}^*\boldsymbol\cA_{\bf n}\}^{-\|\cdot \|}$, we deduce that
$$
\frac{1}{c^2}(id\otimes  \mu_{\bf B})[F(r{\bf S})]\leq (id\otimes  \mu_{\bf A})[F(r{\bf S})]\leq c^2 (id\otimes  \mu_{\bf B})[F(r{\bf S})]
$$
for any $r\in [0,1)$, which implies ${\bf A}\overset{H}{{\underset{c}\sim}}\, {\bf B}$ and completes the proof.
\end{proof}

 A bounded linear operator $A\in B(\cE\otimes \bigotimes_{i=1}^k F^2(H_{n_i}))$ is called {\it $k$-multi-Toeplitz} with respect to the universal model
 ${\bf R}:=({\bf R}_1,\ldots, {\bf R}_k)$,  where  ${\bf R}_i:=({\bf R}_{i,1},\ldots,{\bf R}_{i,n_i})$, if
 $$
 (I_\cE\otimes {\bf R}_{i,s}^*)A(I_\cE\otimes {\bf R}_{i,t})=\delta_{st}A,\qquad s,t\in \{1,\ldots, n_i\},
 $$
  for every $i\in\{1,\ldots, k\}$.
 Let     $\boldsymbol{\cT_{\bf n}}$ be the set of all    $k$-multi-Toeplitz operators  on $\cE\otimes\bigotimes_{i=1}^k F^2(H_{n_i})$.
 In \cite{Po-pluriharmonic-polyball}, we proved that
\begin{equation*}\begin{split}
\boldsymbol{\cT_{\bf n}}&=\text{\rm span} \{f^*g:\ f, g\in B(\cE)\otimes_{min} \boldsymbol\cA_{\bf n}\}^{- \text{\rm SOT}}\\
&=\text{\rm span} \{f^*g:\ f, g\in B(\cE)\otimes_{min} \boldsymbol\cA_{\bf n}\}^{- \text{\rm WOT}},
\end{split}
\end{equation*}
where $\boldsymbol\cA_{\bf n}$ is the polyball algebra.
In what follows, we provide a Harnack type inequality for positive free $k$-pluriharmonic function on the regular polyballs.

 \begin{theorem} \label{H-ineq}
 Let $F$ be a positive free $k$-pluriharmonic function on the regular polyball ${\bf B}_{\bf n}$, with operator coefficients in $B(\cE)$ and let $0\leq r<1$. Then
 $$
 F(0)\left(\frac{1-r}{1+r}\right)^k\leq F({\bf X})\leq F(0)\left(\frac{1+r}{1-r}\right)^k
 $$
 for any ${\bf X}\in r{\bf B}_{\bf n}(\cH)^-$.
 \end{theorem}

 \begin{proof} If $F$ is a  positive free $k$-pluriharmonic function on the regular polyball ${\bf B}_{\bf n}$, with operator coefficients in $B(\cE)$,  then
  there exist   coefficients $A_{(\alpha_1,\ldots,\alpha_k;\beta_1,\ldots, \beta_k)}\in B(\cE)$ with  $\alpha_i,\beta_i\in \FF_{n_i}^+$, $|\alpha_i|=m_i^-, |\beta_i|=m_i^+$
such that, for any $r\in [0,1)$,
$$
F(r{\bf S})= \sum_{m_1\in \ZZ}\cdots \sum_{m_k\in \ZZ} \sum_{{\alpha_i,\beta_i\in \FF_{n_i}^+, i\in \{1,\ldots, k\}}\atop{|\alpha_i|=m_i^-, |\beta_i|=m_i^+}}A_{(\alpha_1,\ldots,\alpha_k;\beta_1,\ldots, \beta_k)}
\otimes r^{\sum_{i=1}^k(|\alpha_i|+|\beta_i|)}
{S}_{1,\alpha_1}{S}_{1,\beta_1}^*\otimes \cdots \otimes {S}_{k,\alpha_k}{ S}_{k,\beta_k}^*,
$$
where the multi-series is convergent in the operator norm topology and the  sum does not depend on the order of the series. In this case,   $F(r{\bf S})$  is a positive $k$-multi-Toeplitz operator and
$$
F(r{\bf S})=F(rS_1,\ldots, rS_k)=\sum_{m_k\in \ZZ}\sum_{|\alpha|=m_k^-, |\beta_k|=m_k^+}
C(\alpha_{k};\beta_k)\otimes r^{|\alpha_k|+|\beta_k|} S_{k,\alpha_k} S_{k, \beta_k}^*.
$$
Note that $F(r{\bf S})$ is also a positive $1$-multi-Topeplitz operator with respect to $R_k=(R_{k,1},\ldots, R_{k, n_k})$, with coefficients  in $B(\cE)\otimes_{min} B(F^2(H_{n_1}))\otimes_{min} \cdots \otimes_{min} B(F^2(H_{n_{k-1}}))$.
Applying the noncommutative Berezin transform at $X_k=(X_{k,1},\ldots, X_{k,n_k})$, we deduce that
$$
G(X_k):=  \sum_{m_k\in \ZZ}\sum_{|\alpha|=m_k^-, |\beta_k|=m_k^+}
C(\alpha_{k};\beta_k)\otimes r^{|\alpha_k|+|\beta_k|} X_{k,\alpha_k} X_{k, \beta_k}^*, \qquad X_k\in [B(\cH)^{n_k}]_1,
$$
is a positive  pluriharmonic function on the unit ball $[B(\cH)^{n_k}]_1$.
Using  the Harnack type inequality from \cite{Po-hyperbolic}, we deduce that
\begin{equation*}
G(0)\frac{1-r}{1+r}\leq G(rS_k)\leq G(0)\frac{1+r}{1-r}.
\end{equation*}
Note that $G(rS_k)=F(rS_1,\ldots, rS_k)$ and $G(0)=F(rS_1,\ldots, rS_{k-1}, 0)$ is a positive $(k-1)$-multi-Toeplitz operator. As above, we obtain
\begin{equation*}
F(rS_1,\ldots, rS_{k-2}, 0,0)\frac{1-r}{1+r}\leq F(rS_1,\ldots, rS_{k-1}, 0)\leq F(rS_1,\ldots, rS_{k-2}, 0,0)\frac{1+r}{1-r}.
\end{equation*}
Continuing this process, we obtain
$$
F(0,\ldots, 0)\frac{1-r}{1+r}\leq F(rS_1,0,\ldots,   0)\leq F(0,\ldots, 0)\frac{1+r}{1-r}.
$$
Now, combining all these inequalities, we deduce that
$$
 F(0)\left(\frac{1-r}{1+r}\right)^k\leq F(rS_1,\ldots, rS_k)\leq F(0)\left(\frac{1+r}{1-r}\right)^k
 $$
 for any $r\in [0,1)$.
 Applying the Berezin transform at  ${\bf X}\in r{\bf B}_{\bf n}(\cH)^-$ to the latter inequalities, we complete the proof.
 \end{proof}

 We define  the {\it free pluriharmonic Poisson kernel}   on the regular polyball  ${\bf B_n}$  by setting
\begin{equation*}
 \boldsymbol{\cP}({\bf R}, {\bf X}):= \sum_{m_1\in \ZZ}\cdots \sum_{m_k\in \ZZ} \sum_{{\alpha_i,\beta_i\in \FF_{n_i}^+, i\in \{1,\ldots, k\}}\atop{|\alpha_i|=m_i^-, |\beta_i|=m_i^+}}{\bf R}_{1,\widetilde\alpha_1}^*\cdots {\bf R}_{k,\widetilde\alpha_k}^*{\bf R}_{1,\widetilde\beta_1}\cdots {\bf R}_{k,\widetilde\beta_k}
\otimes { X}_{1,\alpha_1}\cdots {X}_{k,\alpha_k}{X}_{1,\beta_1}^*\cdots {X}_{k,\beta_k}^*
\end{equation*}
for any ${\bf X}\in {\bf B_n}(\cH)$,
where the convergence is in the operator norm topology, and
  $\widetilde \alpha_i= g^i_{i_k}\cdots g^i_{i_1}$ if
   $\alpha_i=g^i_{i_1}\cdots g^i_{i_k}\in\FF_{n_i}^+$. According to \cite{Po-pluriharmonic-polyball}, the map  ${\bf X}\mapsto
\boldsymbol{\cP}({\bf R}, {\bf X})$ is a positive free $k$-pluriharmonic function
on ${\bf B_n}(\cH)$. In this case, Theorem \ref{H-ineq} implies
 $$
 \left(\frac{1-r}{1+r}\right)^k I\leq \boldsymbol{\cP}({\bf R}, {\bf X})\leq \left(\frac{1+r}{1-r}\right)^k I
 $$
 for any ${\bf X}\in r{\bf B}_{\bf n}(\cH)^-$.

Now, we introduce  a preorder relation
$\overset{P}{\prec} $ on the closed ball ${\bf B_n}(\cH)^-$.
If ${\bf A}$ and ${\bf B}$ are in
${\bf B_n}(\cH)^-$, we say that ${\bf A}$ is {\it Poisson dominated} by ${\bf B}$, and
denote ${\bf A}\overset{P}{\prec}\, {\bf B}$, if there exists $c>0$ such that
$$\boldsymbol{\cP}({\bf R}, r{\bf A})\leq c^2 \boldsymbol{\cP}({\bf R}, r{\bf B}) $$
 for any $r\in [0,1)$.
 When we
want to emphasize the constant $c$, we write
${\bf A}\overset{P}{{\underset{c}\prec}}\, {\bf B}$.
Since $\overset{P}{\prec} $ is a preorder
relation on ${\bf B_n}(\cH)^-$, it induces an equivalence relation
$\overset{P}\sim$ on ${\bf B_n}(\cH)^-$, which we call {\it Poisson
equivalence}. The equivalence classes with respect to
$\overset{P}\sim$ are called {\it Poisson parts} of ${\bf B_n}(\cH)^-$. Let
 ${\bf A}$ and ${\bf B}$ are in
${\bf B_n}(\cH)^-$. It is easy to see that ${\bf A}$ and ${\bf B}$ are Poisson
equivalent (we denote ${\bf A}\overset{P}{\sim}\, {\bf B}$)  if and only if
there exists $c\geq  1$ such that

$$\frac{1}{c^2}\boldsymbol{\cP}({\bf R}, r{\bf B})\leq  \boldsymbol{\cP}({\bf R}, r{\bf A})\leq c^2 \boldsymbol{\cP}({\bf R}, r{\bf B}) $$
 for any $r\in [0,1)$. We also use the
notation ${\bf A}\overset{P}{{\underset{c}\sim}}\, {\bf B}$ if
${\bf A}\overset{P}{{\underset{c}\prec}}\, {\bf B}$ and
${\bf B}\overset{P}{{\underset{c}\prec}}\, {\bf A}$.
We remark that in the particular case when $k=1$, the Poisson equivalence coincides with the Harnack equivalence  (see \cite{Po-hyperbolic}).

We recall that the spectral radius of an  $n_i$-tuple   $A_i:=(A_{i,1},\ldots, A_{i, n_i})$ of operators is given by $$r(A_i):=\lim_{p\to \infty} \|\sum_{\beta_i\in \FF_{n_i}^+, |\beta_i|=p}A_{i,\beta_i} A_{i,\beta_i}^*\|^{1/2p}.
$$
When $n_i=1$ we find again the usual spectral radius of an operator.

\begin{lemma}\label{les}
 Let ${\bf A}=(A_1, \ldots, A_k)\in {\bf B}_{\bf n}(\cH)^-$.
 Then
${\bf A}\overset{P}{\prec}\, {0}$ if and only if  the joint  spectral radius $r(A_i)<1$ for any $i\in \{1,\ldots, k\}$.
\end{lemma}
\begin{proof}
Assume that ${\bf A}\overset{P}{\prec}\, {0}$. Then there is $c>0$ such that
$ \boldsymbol{\cP}({\bf R}, r{\bf A})\leq c^2 I$ for any $r\in [0,1)$.
Set $w:=\sum_{\alpha\in \FF_{n_i}^+} e_\alpha^i\otimes h_\alpha\in F^2(H_{n_i})\otimes \cH$, where $h_\alpha\in \cH$ and $\sum_{\alpha\in \FF_{n_i}^+} \|h_\alpha\|^2<\infty$, and let
$$
\tilde{w}:=\sum_{\alpha\in \FF_{n_i}^+} (\underbrace{1\otimes\cdots \otimes  1}_{\text{${i-1}$
times}}  \otimes\,  e_\alpha^i\otimes 1\otimes \cdots \otimes 1)\otimes h_\alpha
$$
be in $F^2(H_{n_1})\otimes \cdots \otimes F^2(H_{n_k})$. Note that
$$
\left< P(R_i, rA_i)w,w\right>=\left<\boldsymbol{\cP}({\bf R}, r{\bf A})\tilde{w}, \tilde{w}\right>\leq c^2\|\tilde{w}\|^2=c^2\|w\|^2
$$
for any $w\in F^2(H_{n_i})\otimes \cH$, where $P(R_i, rA_i)$ is the Poisson kernel associate with the row contraction $rA_i$.
Applying Theorem 1.2 from \cite{Po-hyperbolic}, we deduce that ${A}_i\overset{P}{\prec}\, {0}$  and, consequently,  $r(A_i)<1$.

Conversely, let ${\bf A}=(A_1, \ldots, A_k)\in {\bf B}_{\bf n}(\cH)^-$ and assume that $r(A_i)<1$ for any $i\in \{1,\ldots, k\}$.
Due to the fact that, for each $i\in \{1,\ldots, k\}$,  ${\bf R}_{i,1},\ldots, {\bf R}_{i,n_i}$ are isometries with orthogonal ranges, we have
  $$
  r(A_i)=r({\bf R}_{i,1}\otimes A_{i,1}^*+\cdots +{\bf R}_{i,n_i}\otimes A_{i,n_i}^*).
  $$
Setting  $\Lambda_i:={\bf R}_{i,1}^*\otimes A_{i,1}+\cdots +{\bf R}_{i,n_i}^*\otimes A_{i,n_i}$ we have  $r(\Lambda_i)=r(A_i)<1$ and  deduce that      the spectrum of
$\Lambda_i$ is included in $ \DD:=\{z\in \CC:\ |z|<1\}$. Since $\lambda\mapsto (\lambda I-\Lambda_i)^{-1}$ is continuous on the resolvent $\rho(\Lambda_i)$, one can easily see that
$M_i:=\sup_{r\in [0,1)} \|(I-r\Lambda_i)^{-1}\|<\infty$.
According to \cite{Po-pluriharmonic-polyball}, $\boldsymbol{\cP}({\bf R}, r{\bf S})$ is equal to
\begin{equation*}
\left[\prod_{i=1}^k(I-{\bf R}_{i,1}^*\otimes r{\bf S}_{i,1}-\cdots
-{\bf R}_{i,n_i}^*\otimes r{\bf S}_{i,n_i})^{-1}\right] \left(I\otimes \boldsymbol\Delta_{r{\bf S}}(I)\right)
\prod_{i=1}^k(I-{\bf R}_{i,1}\otimes r{\bf S}^*_{i,1}-\cdots
-{\bf R}_{i,n_i}\otimes r{\bf S}^*_{i,n_i})^{-1}
\end{equation*}
for any $r\in [0,1)$.
Since the noncommutative Berezin transform
$\boldsymbol\cB_{ {\bf A}}$ is continuous in the operator norm and completely positive, so is $id\otimes \boldsymbol\cB_{ {\bf A}}$. Using the map  $id\otimes \boldsymbol\cB_{ {\bf A}}$,
  we deduce that
\begin{equation*}
\begin{split}
\boldsymbol{\cP}({\bf R}, r{\bf A})
&=\left(id\otimes \boldsymbol\cB_{ {\bf A}}\right)\left[\boldsymbol{\cP}({\bf R}, r{\bf S})\right]\\
&=\prod_{i=1}^k
 \left(I-r\Lambda_i\right)^{-1}\left(I\otimes \boldsymbol\Delta_{r{\bf A}}(I)\right)
 \prod_{i=1}^k
 \left(I-r\Lambda_i^*\right)^{-1}\\
 &\leq M_1^2\cdots M_k^2 I
\end{split}
\end{equation*}
for any $r\in [0,1)$.  This shows that ${\bf A}\overset{P}{\prec}\, {0}$ and completes the proof.
\end{proof}

Now, we show that the Harnack (resp.~Poisson) equivalence class containing the zero element  coincides with the open polyball  ${\bf B}_{\bf n}(\cH)$.

\begin{theorem}
\label{foias2}   Let ${\bf A}=(A_1, \ldots, A_k)\in {\bf B}_{\bf n}(\cH)^-$.
 Then the following statements are equivalent.
 \begin{enumerate}
 \item[(i)]
${\bf A}\overset{H}{\sim}\, 0$;
\item[(ii)] $r(A_i)<1$ for any $i\in \{1,\ldots, k\}$ and there exists $a>0$ such that
    $$\boldsymbol{\cP}({\bf R}, r{\bf A})\geq aI, \qquad r\in [0,1);
    $$
 \item[(iii)]${\bf A}\in {\bf B_n}(\cH)$;
 \item[(iv)]
${\bf A}\overset{P}{\sim}\, 0$.
 \end{enumerate}
 \end{theorem}

\begin{proof} Note that if $F:{\bf B}_{\bf n}(\cH)\to B(\cE)\otimes_{min} B(\cH)$ is a free $k$-pluriharmonic function then, for each $i\in \{1,\ldots, k\}$, the map
$$
{\bf X}_i\mapsto F(\underbrace{0,\ldots, 0}_{\text{${i-1}$
times}}, {\bf X}_i, 0,\ldots, 0), \qquad {\bf X}_i\in [B(\cH)^{n_i}]_1,
$$
is a pluriharmonic function on $[B(\cH)^{n_i}]_1$. Conversely, if $F_i:[B(\cH)^{n_i}]_1\to B(\cE)\otimes_{min} B(\cH)$ is a free pluriharmonic function on the open unit ball $[B(\cH)^{n_i}]_1$, then the map ${\bf X}=(X_1,\ldots, X_k)\mapsto F_i(X_i)$ is a free $k$-pluriharmonic function on ${\bf B_n}(\cH)$.  Consequently, if ${\bf A}=( A_1,\ldots,A_k)$ and $ {\bf B}=( B_1,\ldots,B_k)$ are in the closed  regular polyball ${\bf B}_{\bf n}(\cH)^-$ and
${\bf A} \overset{H}{\sim}\, {\bf B}$, then $A_i\overset{H}{\sim}\, B_i$  for any $i\in \{1,\ldots, k\}$.

Now, assume that ${\bf A}\overset{H}{\sim}\, 0$. Due to the  remark above, we must have ${A}_i\overset{H}{\sim}\, 0$ for any $i\in \{1,\ldots, k\}$. Applying   Theorem 1.6 from \cite{Po-hyperbolic} to $A_i$, we deduce that  $A_i\in [B(\cH)^{n_i}]_1$. In particular,   the joint spectral radius $r(A_i)<1$.
Since ${\bf A}\overset{H}{\sim}\, 0$ and the map  ${\bf X}\mapsto
\boldsymbol{\cP}({\bf R}, {\bf X})$ is a positive free $k$-pluriharmonic function
on ${\bf B_n}(\cH)$, there is $a>0$ such that
$\boldsymbol{\cP}({\bf R}, r{\bf A})\geq aI$ for any $r\in [0,1)$.  Therefore, item (ii) holds.

Now, we prove that (ii)$\implies$(iii). As in the proof of Lemma \ref{les}, we have

\begin{equation*}
\begin{split}
\boldsymbol{\cP}({\bf R}, r{\bf A})=\prod_{i=1}^k
 \left(I-r\Lambda_i\right)^{-1}\left(I\otimes \boldsymbol\Delta_{r{\bf A}}(I)\right)
 \prod_{i=1}^k
 \left(I-r\Lambda_i^*\right)^{-1},
\end{split}
\end{equation*}
where $\Lambda_i:={\bf R}_{i,1}^*\otimes A_{i,1}+\cdots +{\bf R}_{i,n_i}^*\otimes A_{i,n_i}$. Moreover,  the spectral radius of the operator  $\Lambda_i $ coincides with $r(A_i)<1$ and
$$
0<M:=\sup_{r\in [0,1)}\|\prod_{i=1}^k(I-r\Lambda_i)^{-1}\|<\infty.
$$
 Hence and  using the fact that $\boldsymbol{\cP}({\bf R}, r{\bf A})\geq aI$ for any $r\in [0,1)$, we obtain
 $$
I\otimes \boldsymbol\Delta_{r{\bf A}}(I)\geq \frac{c}{M^2} I\otimes I, \qquad r\in [0,1).
$$
Hence $\boldsymbol\Delta_{{\bf A}}(I)\geq \frac{c}{M^2}I$, which shows that ${\bf A}\in {\bf B}_{\bf n}(\cH)$. Therefore, item (iii) holds.
To prove the  implication (iii)$\implies$(i), assume that  ${\bf A}\in {\bf B_n}(\cH)$.
 Define
 $$
 \gamma_{\bf A}=m_{{\bf B}_{\bf n}}({\bf A}):=\inf \{ t>0:\ {\bf A}\in t{\bf B}_{\bf n}(\cH)\}.
 $$
 Due to Proposition 2.6 from \cite{Po-pluriharmonic-polyball}, we have $m_{{\bf B}_{\bf n}}({\bf A})<1$ and, consequently,
   $r{\bf A}\in \gamma_{\bf A}{\bf B_n}(\cH)^-$ for any $r\in [0,1)$. Applying the Harnack type inequality of Theorem \ref{H-ineq}, we obtain the inequalities
$$
 F(0)\left(\frac{1-\gamma_{\bf A}}{1+\gamma_{\bf A}}\right)^k\leq F(r{\bf A})\leq F(0)\left(\frac{1+\gamma_{\bf A}}{1-\gamma_{\bf A}}\right)^k
 $$
 for any  positive free $k$-pluriharmonic function $F$
on ${\bf B_n}(\cH)$  and any $r\in [0,1)$.
Consequently, ${\bf A}\overset{H}{\sim}\, 0$, which completes the proof of the implication (iii)$\implies$(i). Note that if ${\bf A}\overset{H}{\sim}\, 0$, then we automatically have ${\bf A}\overset{P}{\sim}\, 0$. On the other hand,  taking into account Lemma \ref{les}, one can easily see that (iv)$\implies$(ii). This completes the proof.
\end{proof}

\bigskip

\section{Hyperbolic metric on the Harnack parts of the closed  polyball}

In this section, we introduce a hyperbolic  type metric on the Harnack parts of the closed regular polyball and  show that it is invariant under the automorphism group of all free biholomorphic functions of the polyball. We provide a Schwarz-Pick type result for free holomorphic functions on regular polybals with respect to the hyperbolic metric.

Given ${\bf A}, {\bf B}\in {\bf B_n}(\cH)^-$ in the same Harnack part, i.e.
${\bf A}\,\overset{H}{\sim}\, {\bf B}$, we introduce
\begin{equation*}
 \omega_H({\bf A}, {\bf B}):=\inf\left\{ c > 1: \
{\bf A}\,\overset{H}{{\underset{c}\sim}}\, {\bf B}   \right\}.
\end{equation*}

\begin{lemma}\label{beta}
Let $\Delta$ be a Harnack part of ${\bf B_n}(\cH)^-$ and let ${\bf A}, {\bf B}, {\bf C}\in
\Delta$.
Then the following properties hold:
\begin{enumerate}
\item[(i)] $\omega_H({\bf A}, {\bf B})\geq 1$;
\item[(ii)] $\omega_H({\bf A}, {\bf B})=1$ if and only if ${\bf A}={\bf B}$;
\item[(iii)] $\omega_H({\bf A}, {\bf B})=\omega_H({\bf B}, {\bf A})$;
\item[(iv)] $\omega_H({\bf A}, {\bf C})\leq \omega_H({\bf A}, {\bf B}) \omega_H({\bf B}, {\bf C})$.
\end{enumerate}
\end{lemma}
\begin{proof} First, note that (i) and (iii) are consequences of the definition of $\omega_H({\bf A}, {\bf B})$ and relation \eqref{def-H}. If $\omega_H({\bf A}, {\bf B})=1$, then there is a sequence $c_n>1$ with $c_n\to 1$ such that
\begin{equation*}
 \frac{1}{c_n^2}F(r{\bf B})\leq  F(r{\bf A})\leq c_n^2 F(r{\bf B})
\end{equation*}
 for any positive
free $k$-pluriharmonic function $F$  with operator-valued coefficients
and any $r\in [0,1)$.
In particular, since  the map  ${\bf X}\mapsto
\boldsymbol{\cP}({\bf R}, {\bf X})$ is a positive free $k$-pluriharmonic function
on ${\bf B_n}(\cH)$,  we have
$$
\frac{1}{c_n^2} \boldsymbol{\cP}({\bf R}, r{\bf B})\leq \boldsymbol{\cP}({\bf R}, r{\bf A})\leq c_n^2 \boldsymbol{\cP}({\bf R}, r{\bf B})
$$
for any $n\in \NN$ and $ r\in [0,1)$. Since $c_n\to 1$ we deduce that
$\boldsymbol{\cP}({\bf R}, r{\bf A})=\boldsymbol{\cP}({\bf R}, r{\bf A})$ for any $r\in [0,1)$.
For each $i\in \{1,\ldots, k\}$ and $j\in \{1,\ldots, n_i\}$, let $x:=\underbrace{1\otimes\cdots \otimes  1}_{\text{${i-1}$
times}}  \otimes\,  e_j^i\otimes 1\otimes \cdots \otimes 1$ and $y:=1\otimes \cdots \otimes 1$ be in $F^2(H_{n_1})\otimes\cdots \otimes  F^2(H_{n_k})$.
If $\alpha_i, \beta_i\in \FF_{n_i}^+$ with $|\alpha_i|=m_i^-$ , $|\beta_i|=m_i^+$, we have
$\left<{\bf R}_{1,\widetilde\alpha_1}^*\cdots {\bf R}_{k,\widetilde\alpha_k}^*{\bf R}_{1,\widetilde\beta_1}\cdots {\bf R}_{k,\widetilde\beta_k}x,y\right>
=1$  if $ \beta_s=g_0^s$    for $ s\in \{1,\ldots, k\}$,   and $\alpha_s=g_0^s$   for $ s\in \{1,\ldots, k\}\backslash\{i\}$   and  $\alpha_i=g_j^i$, while
$\left<{\bf R}_{1,\widetilde\alpha_1}^*\cdots {\bf R}_{k,\widetilde\alpha_k}^*{\bf R}_{1,\widetilde\beta_1}\cdots {\bf R}_{k,\widetilde\beta_k}x,y\right>
=0$ otherwise.
Consequently, we have
$$
\left<\boldsymbol{\cP}({\bf R}, {\bf X})(x\otimes h), y\otimes \ell\right>=\left<X_{i,j}h,\ell\right>
$$
for any ${\bf X}:=(X_1,\ldots, X_k)\in {\bf B}_{\bf n}(\cH)$ with $X_i=(X_{i,1},\ldots, X_{i,n_i})$ and any $h,\ell\in \cH$. Now, it is clear that relation
 $\boldsymbol{\cP}({\bf R}, r{\bf A})= \boldsymbol{\cP}({\bf R}, r{\bf B})$ implies ${\bf A}={\bf B}$, which shows that relation (ii) holds.

 Due to the definition \eqref{def-H}, for any positive free $k$-pluriharmonic function $F$ with operator-valued coefficients and any $r\in [0,1)$, we have
 \begin{equation*}
 \frac{1}{\omega_H({\bf A}, {\bf B})^2}F(r{\bf B})\leq  F(r{\bf A})\leq \omega_H({\bf A}, {\bf B})^2 F(r{\bf B})
\end{equation*}
and
\begin{equation*}
 \frac{1}{\omega_H({\bf B}, {\bf C})^2}F(r{\bf C})\leq  F(r{\bf B})\leq \omega_H({\bf B}, {\bf C})^2 F(r{\bf C}).
\end{equation*}
Combining these inequalities, we deduce that
\begin{equation*}
 \frac{1}{\omega_H({\bf A}, {\bf B})^2\omega_H({\bf B}, {\bf C})^2}F(r{\bf C})\leq  F(r{\bf A})\leq \omega_H({\bf A}, {\bf B})^2\omega_H({\bf B}, {\bf C})^2 F(r{\bf C}).
\end{equation*}
Hence, we obtain that $\omega_H({\bf A}, {\bf C})\leq \omega_H({\bf A}, {\bf B}) \omega_H({\bf B}, {\bf C})$, which completes the proof.
\end{proof}

Now, we can introduce  a  hyperbolic   type
metric
  on the Harnack parts of ${\bf B_n}(\cH)^-$.

\begin{proposition}\label{delta}
Let $\Delta$ be a Harnack part of ${\bf B_n}(\cH)^-$ and define
$\delta_H:\Delta\times \Delta \to \RR^+$ by setting
\begin{equation*}
 \delta_H({\bf A},{\bf B}):=\ln \omega_H({\bf A}, {\bf B}),\quad {\bf A}, {\bf B}\in \Delta.
\end{equation*}
Then $\delta_H$ is a metric on  $\Delta$.
\end{proposition}
\begin{proof} The result follows from  Lemma \ref{beta}.
\end{proof}

 \begin{theorem} \label{compo} Let $F$ be a  free $k$-pluriharmonic function on the regular polyball ${\bf B}_{\bf n}$ with operator coefficients in $B(\cE)$ and let
 $G=(G_1,\ldots, G_k):{\bf B}_{\bf n}\to {\bf B}_{\bf n}$ be a free holomorphic function  such that each
$G_{i}:[B(\cH)^{n_i}]_1\to [B(\cH)^{n_i}]_1$ is a free holomorphic function. Then $F\circ G$ is a free $k$-pluriharmonic function on  ${\bf B}_{\bf n}$.
 \end{theorem}
 \begin{proof} According to \cite{Po-pluriharmonic-polyball},  $F$ is a free $k$-pluriharmonic function on the regular polyball ${\bf B}_{\bf n}$ with operator coefficients in $B(\cE)$ if and only if
  there exist  coefficients $A_{(\alpha_1,\ldots,\alpha_k;\beta_1,\ldots, \beta_k)}\in B(\cE)$ with  $\alpha_i,\beta_i\in \FF_{n_i}^+$, $|\alpha_i|=m_i^-, |\beta_i|=m_i^+$
such that, for any $r\in [0,1)$,
$$
F(r{\bf S})= \sum_{m_1\in \ZZ}\cdots \sum_{m_k\in \ZZ} \sum_{{\alpha_i,\beta_i\in \FF_{n_i}^+, i\in \{1,\ldots, k\}}\atop{|\alpha_i|=m_i^-, |\beta_i|=m_i^+}}A_{(\alpha_1,\ldots,\alpha_k;\beta_1,\ldots, \beta_k)}
\otimes r^{\sum_{i=1}^k(|\alpha_i|+|\beta_i|)}
{\bf S}_{1,\alpha_1}\cdots {\bf S}_{k,\alpha_k}{\bf S}_{1,\beta_1}^*\cdots {\bf S}_{k,\beta_k}^*
$$
where the multi-series is convergent in the operator norm topology and its sum does not depend on the order of the series. In this case,   $F(r{\bf S})$  is a $k$-multi-Toeplitz operator   in
$$
\text{\rm span} \{f^*g:\ f, g\in B(\cE)\otimes_{min} \boldsymbol\cA_{\bf n}\}^{-\|\cdot\|},
$$
where   $\boldsymbol\cA_{\bf n}$ is the polyball algebra.

Let $G=(G_1,\ldots, G_k):{\bf B}_{\bf n}\to {\bf B}_{\bf n}$  be a free holomorphic function with $G_i=(G_{i,1},\ldots, G_{i, n_i})$, where each
$G_{i}:[B(\cH)^{n_i}]_1\to [B(\cH)^{n_i}]_1$ is a free holomorphic function.
According to Proposition 2.2 from \cite{Po-automorphisms-polyball},
$\text{\rm range} \, G\subseteq {\bf B}_{\bf n}(\cH)$ if and only if
$G(r{\bf S})\in {\bf B}_{\bf n}(\otimes_{i=1}^k F^2(H_{n_i}))$ for any $r\in [0,1)$, where ${\bf S}=({\bf S}_1,\ldots, {\bf S}_n)$,  ${\bf S}_i=({\bf S}_{i,1},\ldots, {\bf S}_{i,n_i})$,  is the universal model of the regular polyball ${\bf B}_{\bf n}$. Consequently, $F(G(r{\bf S}))$ is equal to
\begin{equation*}
\begin{split}
  \sum_{m_1\in \ZZ}\cdots \sum_{m_k\in \ZZ} \sum_{{\alpha_i,\beta_i\in \FF_{n_i}^+, i\in \{1,\ldots, k\}}\atop{|\alpha_i|=m_i^-, |\beta_i|=m_i^+}}A_{(\alpha_1,\ldots,\alpha_k;\beta_1,\ldots, \beta_k)}
\otimes
G_{1,\alpha_1}(r{\bf S}_1)\cdots {G}_{k,\alpha_k}(r{\bf S}_k){G}_{1,\beta_1}(r{\bf S}_1)^*\cdots {G}_{k,\beta_k}(r{\bf S}_k)^*,
\end{split}
\end{equation*}
where the multi-series  is convergent in the operator norm topology.
  Now, we prove that $F(G(r{\bf S}))$ is a $k$-multi-Toeplitz operator  with respect to the universal model ${\bf R}$.   It is enough to show that
  $$
  T:=G_{1,\alpha_1}(r{\bf S}_1)\cdots {G}_{k,\alpha_k}(r{\bf S}_k){G}_{1,\beta_1}(r{\bf S}_1)^*\cdots {G}_{k,\beta_k}(r{\bf S}_k)^*
  $$
  is $k$-multi-Toeplitz operator  with respect to  ${\bf R}$, when
   $\alpha_i,\beta_i\in \FF_{n_i}^+$, $|\alpha_i|=m_i^-, |\beta_i|=m_i^+$.
   First, note that, for each $i\in \{1,\ldots, k\}$ and $s,t\in \{1,\ldots, n_i\}$
   \begin{equation*}
   {\bf R}_{i,s}^*G_{i,\alpha_i}(r{\bf S}_i) G_{i,\beta_i}(r{\bf S}_i)^* {\bf R}_{i,t}=\delta_{st}G_{i,\alpha_i}(r{\bf S}_i) G_{i,\beta_i}(r{\bf S}_i)^*.
    \end{equation*}
 Hence, we deduce that
 \begin{equation*}
 \begin{split}
 {\bf R}_{i,s}^*T {\bf R}_{i,t}&=\left[\prod_{p\in \{1,\ldots, k\}, p\neq i} G_{p,\alpha_p}(r{\bf S}_p)\right] {\bf R}_{i,s}^*G_{i,\alpha_i}(r{\bf S}_i) G_{i,\beta_i}(r{\bf S}_i)^* {\bf R}_{i,t}\left[\prod_{p\in \{1,\ldots, k\}, p\neq i} G_{p,\alpha_p}(r{\bf S}_p)\right]^*\\
 &=\delta_{st} T
 \end{split}
 \end{equation*}
 for each $i\in \{1,\ldots, k\}$, which proves our assertion.

 Using  Theorem 1.5   from \cite{Po-pluriharmonic-polyball} we deduce that, for each $r\in [0,1)$,   $F(G(r{\bf S}))$   has a unique Fourier representation
 $$
\varphi({\bf S}):= \sum_{m_1\in \ZZ}\cdots \sum_{m_k\in \ZZ} \sum_{{\sigma_i,\omega_i\in \FF_{n_i}^+, i\in \{1,\ldots, k\}}\atop{|\sigma_i|=m_i^-, |\omega_i|=m_i^+}}B^{(r)}_{(\sigma_1,\ldots,\sigma_k;\omega_1,\ldots, \omega_k)}
\otimes r^{\sum_{i=1}^k(|\sigma_i|+|\omega_i|)}
{\bf S}_{1,\sigma_1}\cdots {\bf S}_{k,\sigma_k}{\bf S}_{1,\omega_1}^*\cdots {\bf S}_{k,\omega_k}^*
$$
 and
 $$
\varphi(t{\bf S}):= \sum_{m_1\in \ZZ}\cdots \sum_{m_k\in \ZZ} \sum_{{\sigma_i,\omega_i\in \FF_{n_i}^+, i\in \{1,\ldots, k\}}\atop{|\sigma_i|=m_i^-, |\omega_i|=m_i^+}}B^{(r)}_{(\sigma_1,\ldots,\sigma_k;\omega_1,\ldots, \omega_k)}
\otimes (rt)^{\sum_{i=1}^k(|\sigma_i|+|\omega_i|)}
{\bf S}_{1,\sigma_1}\cdots {\bf S}_{k,\sigma_k}{\bf S}_{1,\omega_1}^*\cdots {\bf S}_{k,\omega_k}^*,
$$
  is convergent    in the operator norm topology for any $t\in [0,1)$.  Moreover
  $F(G(r{\bf S}))=\text{\rm SOT-}\lim_{t\to 1} \varphi(t{\bf S})$ and
 \begin{equation}\label{BFG}
 r^{\sum_{i=1}^k(|\sigma_i|+|\omega_i|)}\left<B^{(r)}_{(\sigma_1,\ldots,\sigma_k;
 \omega_1,\ldots, \omega_k)}h,\ell\right>=\left<F(G((r{\bf S}))(h\otimes x), \ell\otimes y\right>,\qquad  h,\ell\in \cE,
\end{equation}
  where $x:=x_1\otimes \cdots \otimes x_k$, $y=y_1\otimes \cdots \otimes y_k$ with \begin{equation*}
\begin{cases} x_i=e^i_{\omega_i} \text{ and } y_i=1&; \quad \text{if } m_i\geq 0\\
 x_i=1 \text{ and } y_i=e^i_{\sigma_i}&; \quad \text{if } m_i<0
 \end{cases}
 \end{equation*}
 for every $i\in \{1,\ldots, k\}$.
We need to show that each coefficient
$B^{(r)}_{(\sigma_1,\ldots,\sigma_k; \omega_1,\ldots, \omega_k)}$
 does not depend on $r\in [0,1)$. Indeed, using the relations above, we deduce that
 \begin{equation}\label{Prod}
 \begin{split}
 &\left<F(G((r{\bf S}))(h\otimes x), \ell\otimes y\right>\\
 &\qquad =
  \sum_{m_1\in \ZZ}\cdots \sum_{m_k\in \ZZ} \sum_{{\alpha_i,\beta_i\in \FF_{n_i}^+, i\in \{1,\ldots, k\}}\atop{|\alpha_i|=m_i^-, |\beta_i|=m_i^+}}\left<A_{(\alpha_1,\ldots,\alpha_k;\beta_1,\ldots, \beta_k)}h,\ell\right>
   \prod_{i=1}^k \left<G_{i,\beta_i}(r{\bf S}_i)^* x_i, G_{i,\alpha_i}(r{\bf S}_i)^* y_i\right>.
 \end{split}
 \end{equation}
On the other hand, note that if $m_i\geq 0$, then $x_i=e_{\omega_i}^i$ and $y_i=1$. Consequently, we have
\begin{equation*}
 \begin{split}
\left<G_{i,\beta_i}(r{\bf S}_i)^* x_i, G_{i,\alpha_i}(r{\bf S}_i)^* y_i\right>
&=\left<G_{i,\beta_i}(r{\bf S}_i)^* x_i, \overline{G_{i,\alpha_i}(0)} \right>\\
&=G_{i,\alpha_i}(0)\left<e_{\omega_i}^i, G_{i,\beta_i}(r{\bf S}_i)1\right>\\
&=r^{|\omega_i|} M(\alpha_i, \beta_i, \omega_i),
\end{split}
 \end{equation*}
where $M(\alpha_i, \beta_i, \omega_i)$ is a constant which does not depend on $r$.
 Similarly, if  $m_i<0$, we deduce that
 $$\left<G_{i,\beta_i}(r{\bf S}_i)^* x_i, G_{i,\alpha_i}(r{\bf S}_i)^* y_i\right>
 =r^{|\sigma_i|} M(\alpha_i, \beta_i, \sigma_i),
 $$
where $M(\alpha_i, \beta_i, \sigma_i)$ is a constant which does not depend on $r$.
Now, using relations \eqref{BFG} and \eqref{Prod}, one can see that  each coefficient
$B^{(r)}_{(\sigma_1,\ldots,\sigma_k; \omega_1,\ldots, \omega_k)}$
 does not depend on $r\in [0,1)$. Therefore we can write
 $B_{(\sigma_1,\ldots,\sigma_k; \omega_1,\ldots, \omega_k)}=B^{(r)}_{(\sigma_1,\ldots,\sigma_k; \omega_1,\ldots, \omega_k)}$.
Since, due to the considerations above, the multi-series
 $$
  \sum_{m_1\in \ZZ}\cdots \sum_{m_k\in \ZZ} \sum_{{\sigma_i,\omega_i\in \FF_{n_i}^+, i\in \{1,\ldots, k\}}\atop{|\sigma_i|=m_i^-, |\omega_i|=m_i^+}}B_{(\sigma_1,\ldots,\sigma_k;\omega_1,\ldots, \omega_k)}
\otimes (rt)^{\sum_{i=1}^k(|\sigma_i|+|\omega_i|)}
{\bf S}_{1,\sigma_1}\cdots {\bf S}_{k,\sigma_k}{\bf S}_{1,\omega_1}^*\cdots {\bf S}_{k,\omega_k}^*,
$$
  is convergent    in the operator norm topology for any $t, r\in [0,1)$, we conclude that  $F\circ G$ is a free $k$-pluriharmonic function.
  The proof is complete.
\end{proof}

 The following result is a Schwarz-Pick lemma for free holomorphic functions on the regular polyball ${\bf B_n}$
with operator-valued coefficients, with respect to the hyperbolic
metric.

\begin{theorem} \label{S-P} Let ${\bf G}:{\bf B_n}(\cH)\to{\bf B_n}(\cH)^-$ be  free holomorphic functions such that
$${\bf G}({\bf X}):=(G_1({X}_1),\ldots, G_k({ X}_k))\in {\bf B_n}(\cH)^-,\qquad
{\bf X}:=(X_1,\ldots, X_k)\in {\bf B_n}(\cH),
$$
where
  $G_i:=(G_{i,1},\ldots, G_{i,n_i}):[B(\cH)^{n_i}]_1\to [B(\cH)^{n_i}]_1^-$ and each $G_{i,j}:[B(\cH)^{n_i}]_1\to B(\cH)$ is a free holomorphic function on $[B(\cH)^{n_i}]_1$. If ${\bf X}, {\bf Y}\in {\bf B_n}(\cH)$,
then $G({\bf X})\,\overset{H}{{ \sim}}\, G({\bf Y})$ and
$$
\delta_H({\bf G}({\bf X}), {\bf G}({\bf Y}))\leq \delta_H({\bf X}, {\bf Y}),
$$
where $\delta_H$ is the hyperbolic  metric  defined on  the Harnack parts of  the closed   polyball ${\bf B_n}(\cH)^-$.
\end{theorem}
\begin{proof}
Let $F:{\bf B}_{\bf n}(\cH)\to B(\cE)\otimes_{min} B(\cH)$ be a positive free $k$-pluriharmonic function on ${\bf B}_{\bf n}(\cH)$ with coefficients in $B(\cE)$.
 Due to Theorem \ref{compo}, the map
 $F\circ \gamma {\bf G}$  is a positive  free $k$-pluriharmonic function on ${\bf B}_{\bf n}(\cH)$ for any $\gamma\in [0,1)$.
 If   ${\bf X}, {\bf Y}\in {\bf B_n}(\cH)$, Theorem \ref{foias2} shows that ${\bf A}\,\overset{H}{{\underset{c}\sim}}\, {\bf B}$ for some $c\geq 1$. Consequently,
 $$
 \frac{1}{c^2}(F\circ \gamma{ \bf G})(r{\bf Y})\leq  (F\circ \gamma{ \bf G})(r{\bf X})\leq c^2 (F\circ \gamma{ \bf G})(r{\bf Y})
 $$
 for any $\gamma,r\in [0,1)$. Taking $r\to 1$ and using the continuity of $F\circ \gamma {\bf G}$ on ${\bf B_n}(\cH)$  in the operator norm topology, we deduce that $$
 \frac{1}{c^2}(F\circ \gamma{ \bf G})({\bf Y})\leq  (F\circ \gamma{ \bf G})({\bf X})\leq c^2 (F\circ \gamma{ \bf G})({\bf Y})
 $$
 for any $\gamma\in [0,1)$. Hence ${ \bf G}({\bf X})\overset{H}{{\underset{c}\prec}}\, { \bf G}({\bf Y})$. Using the definition of the hyperbolic  metric  defined on  the Harnack parts of  the closed   polyball ${\bf B_n}(\cH)^-$, one can complete the proof.
\end{proof}

 The next result is a    Schwarz-Pick lemma for free holomorphic functions from  the regular polyball ${\bf B_n}(\cH)$ to the unit ball $[B(\cH)^m]_1$,
  with respect to the hyperbolic
metric.

\begin{proposition}\label{SP1} Let $\Phi=(\Phi_1,\ldots, \Phi_m):{\bf B}_{\bf n} (\cH)\to   [B(\cH)^m]_1^-$  be a  free holomorphic function on the regular polyball.
If ${\bf X}, {\bf Y}\in {\bf B_n}(\cH)$,
then $\Phi({\bf X})\,\overset{H}{{ \sim}}\, \Phi({\bf Y})$ and
$$
\delta_H(\Phi({\bf X}), \Phi({\bf Y}))\leq \delta_H({\bf X}, {\bf Y}),
$$
where $\delta_H$ is the hyperbolic  metric  defined on  the Harnack parts of $[B(\cH)^m]_1^-$ and  on the polyball ${\bf B_n}(\cH)$, respectively.
\end{proposition}
\begin{proof}
Let $F:[B(\cH)^m]_1\to B(\cE)\otimes_{min} B(\cH)$ be a positive free pluriharmonic function. According to \cite{Po-pluriharmonic}, there are some operators $C_{(\alpha)}\in B(\cE)$, $\alpha\in \FF_m^+$, such that
$$
F(Y_1,\ldots, Y_m)=\sum_{q=1}^\infty \sum_{\alpha\in \FF_m^+, |\alpha|=q} C_{(\alpha)}^*\otimes Y_\alpha^*+\sum_{q=0}^\infty \sum_{\alpha\in \FF_m^+, |\alpha|=q} C_{(\alpha)}\otimes Y_\alpha, \qquad (Y_1,\ldots, Y_m)\in [B(\cH)^m]_1,
$$
where the convergence is in the operator norm topology.
 If $G=(G_1,\ldots, G_m):{\bf B}_{\bf n} (\cH)\to [B(\cH)^m]_1$ is a free holomorphic function  on the regular polyball, i.e. each $G_j:{\bf B_n}(\cH)\to B(\cH)$ is free holomorphic, then,  using Theorem 2.4 from \cite{Po-automorphisms-polyball}, we deduce that  $F\circ G$ is a  positive free $k$-pluriharmonic function  on ${\bf B}_{\bf n} (\cH)$.
 Applying this result when $G_j=r\Phi_j$, $r\in [0,1)$, and $j\in \{1,\ldots, m\}$, we have that $F\circ r\Phi$ is a positive free $k$-pluriharmonic function. Now, the rest of the proof is similar to that of Theorem  \ref{S-P}. We leave it to the reader.
\end{proof}

Let  ${\bf n}=(n_1,\ldots, n_k)\in \NN^k$
 and let $\sigma$ be  a permutation of the set  $\{1,\ldots, k\}$ such that $n_{\sigma(i)}=n_i$.  Then  the map $p_\sigma:{\bf B_n}(\cH)^-\to {\bf B_n}(\cH)^-$, defined by
 $$p_\sigma({\bf X} )=(X_{\sigma (1)},\ldots, X_{\sigma(k)}), \qquad {\bf X}:=(X_1,\ldots, X_k)\in {\bf B_n}(\cH)^-,$$
 is a homeomorphism of ${\bf B_n}(\cH)^-$ and $p_\sigma|_{{\bf B_n}(\cH)}$ a free holomorphic automorphism of  ${\bf B_n}(\cH)$.
 If each $U_i\in \CC^{n_i}$, $i\in \{1,\ldots, k\}$,  is a  unitary operator and    ${\bf U}\in B(C^{n_1+\cdots +n_k})$ is   the direct sum ${\bf U}=U_1\oplus \cdots \oplus U_k$, then  the map $\boldsymbol\Phi_{\bf U}:{\bf B_n}(\cH)^-\to
{\bf B_n}(\cH)^-$ defined by  $\boldsymbol\Phi_{\bf U}({\bf X}):={\bf X} {\bf U}$  is also a  free holomorphic  automorphism of ${\bf B_n}(\cH)$ and homeomorphism of ${\bf B_n}(\cH)^-$.

  In \cite{Po-automorphisms-polyball}, we obtained a complete description of the group $\text{\rm Aut}({\bf B_n})$ of all free holomorphic automorphisms of the regular polyball ${\bf B_n}$.  More precisely, we proved that if $\boldsymbol\Psi\in \text{\rm Aut}({\bf B_n})$ and  $\boldsymbol\lambda=(\lambda_1,\ldots, \lambda_k):=\boldsymbol\Psi^{-1}(0)$, then there are  unique unitary operators $U_i\in B(\CC^{n_i})$, $i\in \{1,\ldots, k\}$, and a unique permutation $\sigma\in \cS_k$ with $n_{\sigma(i)}=n_i$  such that
 $$
 \boldsymbol\Psi=p_\sigma\circ \boldsymbol\Phi_{\bf U}\circ \boldsymbol\Psi_{\boldsymbol\lambda},
 $$
 where ${\bf U}:=U_1\oplus\cdots \oplus U_k$,  $\boldsymbol\Psi_{\boldsymbol\lambda}:=(\Psi_{\lambda_1},\ldots, \Psi_{\lambda_k})$, and $\Psi_{\lambda_i}$ is the involutive free holomorphic automorphisms of the open unit ball $[B(\cH)^{n_i}]_1$ (see \cite{Po-automorphism}). Moreover, we showed that  if
  $\hat{\bf \Psi}=(\hat{\Psi}_1,\ldots, \hat{\Psi}_k)$ is
the boundary function with respect to the universal model ${\bf S}=\{{\bf S}_{i,j}\}$, i.e.  $\hat{\bf \Psi}:=\lim_{r\to 1}{\bf \Psi}(r{\bf S})$, then the following statements hold:
\begin{enumerate}
\item[(i)] ${\bf \Psi}$ is
 a free holomorphic function on the regular polyball $\gamma {\bf B_n}$ for some $\gamma>1$.
\item[(ii)]  $\hat{\bf \Psi}$   is a pure element in the polyball $ {\bf B_n}(\otimes_{i=1}^k F^2(H_{n_i}))^-$ and $\hat{\bf \Psi} ={\bf \Psi}({\bf S})$. Each $\hat{\Psi}_i=(\hat{\Psi}_{i,1},\ldots, \hat{\Psi}_{i,n_i})$  is an  isometry with entries in the noncommutative disk algebra generated by ${\bf S}_{i,1},\ldots, {\bf S}_{i,n_i}$ and the identity.

\item[(iii)]${\bf \Psi}$ is a homeomorphism of ${\bf B_n}(\cH)^-$ onto ${\bf B_n}(\cH)^-$.
    \end{enumerate}

In what follows we show that the hyperbolic metric is invariant under  the group $Aut({\bf B_n})$ of all free holomorphic automorphisms of ${\bf B_n}$.

\begin{theorem}\label{formula}
 Let
 ${\bf A}$ and ${\bf B}$ be  in
${\bf B_n}(\cH)^-$ such that ${\bf A}\overset{H}{\sim}\, {\bf B}$. Then
 $$
\delta_H({\bf A}, {\bf B})=\delta_H(\boldsymbol\Psi({\bf A}), \boldsymbol\Psi({\bf B})), \qquad \boldsymbol\Psi\in
Aut({\bf B_n}).
$$
\end{theorem}
\begin{proof}
  Let
 ${\bf A}$ and ${\bf B}$ be  in
${\bf B_n}(\cH)^-$ and let $\boldsymbol\Psi\in Aut({\bf B_n})$. First, we prove that if $c\geq 1$, then
${\bf A}\overset{H}{{\underset{c}\prec}}\, {\bf B}$ if and only if
$\boldsymbol\Psi( {\bf A})\overset{H}{{\underset{c}\prec}}\,
\boldsymbol\Psi({\bf B})$. Assume that ${\bf A}\overset{H}{{\underset{c}\prec}}\, {\bf B}$
and let $F:{\bf B}_{\bf n}(\cH)\to B(\cE)\otimes_{min} B(\cH)$ be a positive free $k$-pluriharmonic function on ${\bf B}_{\bf n}(\cH)$.
 Due to Theorem \ref{compo} and the remarks above, the map
 $F\circ (\gamma\boldsymbol\Psi)$  is a positive  free $k$-pluriharmonic function on ${\bf B}_{\bf n}(\cH)$ for any $\gamma\in [0,1)$. If ${\bf A}\overset{H}{{\underset{c}\prec}}\, {\bf B}$, then
 $$
 (F\circ \gamma\boldsymbol\Psi)(r{\bf A})\leq c^2 (F\circ \gamma\boldsymbol\Psi)(r{\bf B})
 $$
 for any $\gamma,r\in [0,1)$. Taking $r\to 1$ in the inequality above and using the fact that  $\gamma  \boldsymbol\Psi(r{\bf A})\to \gamma \boldsymbol\Psi({\bf A})$ and
 $(F\circ \gamma\boldsymbol\Psi)(r{\bf A})\to (F\circ \gamma\boldsymbol\Psi)({\bf A})$ in the operator norm topology, we deduce that
$$
 F( \gamma\boldsymbol\Psi({\bf A}))\leq c^2 F (\gamma\boldsymbol\Psi({\bf B}))
 $$
for any $\gamma\in [0,1)$, which shows that
$\Psi( {\bf A})\overset{H}{{\underset{c}\prec}}\,
\Psi({\bf B})$. Conversely, if $\boldsymbol\Psi( {\bf A})\overset{H}{{\underset{c}\prec}}\,
\boldsymbol\Psi({\bf B})$, then  applying the direct implication  to $\boldsymbol\Psi^{-1}$, we obtain that $\boldsymbol\Psi^{-1}(\Psi( {\bf A}))\overset{H}{{\underset{c}\prec}}\,
\boldsymbol\Psi^{-1}(\boldsymbol\Psi({\bf B}))$. Since $\boldsymbol\Psi^{-1}\circ\boldsymbol\Psi=id $ on ${\bf B_n}(\cH)^-$, we deduce that    ${\bf A}\overset{H}{{\underset{c}\prec}}\, {\bf B}$.
Now, it is clear that ${\bf A}\,\overset{H}{{\underset{c}\sim}}\, {\bf B} $
if and only if $\boldsymbol\Psi({\bf A})\,\overset{H}{{\underset{c}\sim}}\, \boldsymbol\Psi({\bf B})$, which shows that $\delta_H({\bf A}, {\bf B})=\delta_H(\boldsymbol\Psi({\bf A}), \boldsymbol\Psi({\bf B}))$ for any $ \boldsymbol\Psi\in
Aut({\bf B_n})$.
The proof is complete.
\end{proof}

Fix  ${\bf n}=(n_1,\ldots, n_k)\in \NN^k$
 and let $\sigma$ be  a permutation of the set  $\{1,\ldots, k\}$ such that $n_{\sigma(i)}=n_i$.
Let $\cG$ be the set of all free holomorphic functions  of the form
 $p_\sigma\circ{\bf G}:{\bf B_n}(\cH)\to B(\cH)^{n_1+\cdots +n_k}$  where ${\bf G}$ has the form
$${\bf G}({\bf X}):=(G_1({X}_1),\ldots, G_k({ X}_k))\in {\bf B_n}(\cH),\qquad
{\bf X}:=(X_1,\ldots, X_k)\in {\bf B_n}(\cH),
$$
with
  $G_i:=(G_{i,1},\ldots, G_{i,n_i}):[B(\cH)^{n_i}]_1\to B(\cH)^{n_i}$ and each $G_{i,j}:[B(\cH)^{n_i}]_1\to B(\cH)$ is a free holomorphic function on $[B(\cH)^{n_i}]_1$.
  Note that  the group $\text{\rm Aut}({\bf B_n})$ of all free holomorphic automorphisms of the regular polyball ${\bf B_n}$ is included in $\cG$.

In what follows we define  a  Kobayashi type pseudo-distance on domains
$M\subset B(\cH)^{n_1+\cdots +n_k}$   with respect to the hyperbolic
metric $\delta_H$ of the regular polyball  ${\bf B_n}(\cH)$. Given two
points ${\bf X},{\bf Y}\in M$, we consider a {\it chain of free holomorphic
polyballs} from ${\bf X}$ to ${\bf Y}$. That is, a chain of elements
$$
{\bf X}={\bf X}_0,{\bf X}_1,\ldots, {\bf X}_m={\bf Y}
$$ in $M$, pairs  $({\bf A}^{(1)},
{\bf B}^{(1)}),\ldots, ({\bf A}^{(m)},{\bf B}^{(m)})$ of elements in  ${\bf B_n}(\cH)$, and free holomorphic
functions $F_1,\ldots, F_m$ in $\cG$ with values  in $M$
such that
$$
F_j({\bf A}^{(j)})={\bf X}^{(j-1)}\ \text{ and } \ F_j({\bf B}^{(j)})={\bf X}^{(j)}\ \text{ for } \
j=1,\ldots, m.
$$
Denote this  chain by $\gamma$ and define its length by
$$
\ell(\gamma):=\delta_H({\bf A}^{(1)},
{\bf B}^{(1)})+\cdots + \delta ({\bf A}^{(m)},{\bf B}^{(m)}),
$$
where $\delta_H$ is the hyperbolic metric  on ${\bf B_n}(\cH)$. We
define the Kobayashi type pseodo-distance
$$
\delta_{\bf B_n}^M(X,Y):=\inf \ell(\gamma),
$$
where the infimum is taken over all chains $\gamma$ of free
holomorphic polyballs from ${\bf X}$ to ${\bf Y}$. If there is no such chain, we set
$\delta_{\bf B_n}^M({\bf X},{\bf Y})=\infty$. In general,  $\delta_{\bf B_n}^M$ is not  a true  distance  on $M$.

\begin{proposition}\label{Koba} If $M={\bf B_n}(\cH)$, then
$\delta_{\bf B_n}^M$ is a true distance and \ $ \delta_{\bf B_n}^M=\delta_H. $
\end{proposition}
\begin{proof} Fix  ${\bf X},{\bf Y}\in {\bf B_n}(\cH)$ and let
 $\gamma$ be a chain of free holomorphic
polyballs from ${\bf X}$ to ${\bf Y}$, as described above.
Since  $\delta_H$ is a metric, Theorem \ref{S-P}  implies
\begin{equation*}
\begin{split}
\delta_H({\bf X},{\bf Y})&\leq  \delta_H({\bf X}_0,{\bf X}_1)+ \cdots +
\delta_H({\bf X}_{m-1},{\bf X}_m)\\
&=
 \delta_H(F_1({\bf A}^{(1)}),
F_1({\bf B}^{(1)}) +\cdots +
 \delta_H(F_k({\bf A}^{(m)}),
F_k({\bf B}^{(m)})\\
&\leq \delta_H({\bf A}^{(1)},
{\bf B}^{(1)})+\cdots + \delta ({\bf A}^{(m)},{\bf B}^{(m)})=\ell(\gamma).
\end{split}
\end{equation*}
Taking the infimum over   all chains $\gamma$ of free holomorphic
polyballs from ${\bf X}$ to ${\bf Y}$, we deduce that $\delta_H({\bf X},{\bf Y})\leq \delta_{\bf B_n}^M({\bf X},{\bf Y})$. Taking $F$  the identity on ${\bf B_n}(\cH)$, we obtain
$\delta_H({\bf X},{\bf Y})= \delta_{\bf B_n}^M({\bf X},{\bf Y})$. The proof is complete.
\end{proof}

We remark that, in the particular case when $k=1$ and $n_1=1$, Proposition \ref{Koba} implies   the  well-known result that the Kobayashi distance on the open unit disc
$\DD$ coincides with the Poincar\' e  metric.    On the other hand, if $k=1$ and $n_1\in \NN$, we find again a result from \cite{Po-hyperbolic}.

\bigskip

\section{A metric   on the  Poisson parts of  the closed polyball}

In this section, we introduce the Poisson metric $\delta_\cP$ on Poisson parts of the closed polyball and  obtain an explicit  formula for $\delta_\cP$    in terms of certain noncommutative Cauchy kernels acting on tensor products of full Fock spaces.
We also prove that $\delta_\cP$ is a complete metric on ${\bf B_n}(\cH)$ and that the $\delta_\cP$-topology coincides with the operator norm topology on ${\bf B_n}(\cH)$.

Given ${\bf A}, {\bf B}\in {\bf B_n}(\cH)^-$ in the same Poisson part, i.e.
${\bf A}\,\overset{P}{\sim}\, {\bf B}$, we introduce
\begin{equation*}
 \omega_\cP({\bf A}, {\bf B}):=\inf\left\{ c > 1: \
{\bf A}\,\overset{P}{{\underset{c}\sim}}\, {\bf B}   \right\}.
\end{equation*}

\begin{lemma}\label{omega}
Let $\Delta$ be a Poisson part of ${\bf B_n}(\cH)^-$ and let ${\bf A}, {\bf B}, {\bf C}\in
\Delta$.
Then the following properties hold:
\begin{enumerate}
\item[(i)] $\omega_\cP({\bf A}, {\bf B})\geq 1$;
\item[(ii)] $\omega_\cP({\bf A}, {\bf B})=1$ if and only if ${\bf A}={\bf B}$;
\item[(iii)] $\omega_\cP({\bf A}, {\bf B})=\omega_\cP({\bf B}, {\bf A})$;
\item[(iv)] $\omega_\cP({\bf A}, {\bf C})\leq \omega_\cP({\bf A}, {\bf B}) \omega_\cP({\bf B}, {\bf C})$.
\end{enumerate}
\end{lemma}
\begin{proof}
The proof is similar to Lemma \ref{beta}.
\end{proof}

Now, we can introduce  a
metric
  on the Poisson parts of ${\bf B_n}(\cH)^-$.

\begin{proposition}\label{delta2}
Let $\Delta$ be a Poisson part of ${\bf B_n}(\cH)^-$ and define the function
$\delta_\cP:\Delta\times \Delta \to \RR^+$ by setting
\begin{equation*}
 \delta_\cP({\bf A},{\bf B}):=\ln \omega_\cP({\bf A}, {\bf B}),\quad {\bf A}, {\bf B}\in \Delta.
\end{equation*}
Then $\delta_\cP$ is a metric on  $\Delta$.
\end{proposition}
\begin{proof} The result follows from  Lemma \ref{omega}.
\end{proof}

\begin{theorem}\label{dp} If ${\bf A}$ and ${\bf B}$ are   in the open ball
${\bf B_n}(\cH)$, then
$$
\delta_{\cP}({\bf A}, {\bf B})=\ln \max \left\{\left\|C_{\bf A}({\bf R}) C_{\bf B}({\bf R})^{-1}\right\|, \left\|C_{\bf B}({\bf R})C_{\bf A}({\bf R})^{-1}\right\|\right\},
$$
where
 $$C_{\bf X}({\bf R}):=(I\otimes \boldsymbol\Delta_{\bf X}(I)^{1/2})\prod_{i=1}^k(I-{\bf R}_{i,1}\otimes
X_{i,1}^*-\cdots -{\bf R}_{i,n_i}\otimes X_{i,n_i}^*)^{-1}$$ for any
${\bf X}=(X_1,\ldots, X_k)\in {\bf B_n}(\cH)$  with $X_i=(X_{i,1},\ldots, X_{i, n_i})$.
\end{theorem}
\begin{proof}
Due to Theorem \ref{foias2}, the open polyball ${\bf B_n}(\cH)$ is the Poisson part of ${\bf B_n}(\cH)^-$ containing the zero element. Let ${\bf A},{\bf B} \in{\bf B_n}(\cH)$ and assume that
${\bf A}\,\overset{P}{{\underset{c}\sim}}\, {\bf B} $ for some $c\geq 1$. Then
\begin{equation}
\label{CB}
\frac{1}{c^2}\boldsymbol{\cP}({\bf R}, r{\bf B})\leq  \boldsymbol{\cP}({\bf R}, r{\bf A})\leq c^2 \boldsymbol{\cP}({\bf R}, r{\bf B})
\end{equation}
 for any $r\in [0,1)$. According to Theorem 4.2 from \cite{Po-pluriharmonic-polyball}, we have
 $$
 \boldsymbol{\cP}({\bf R}, {\bf X})=C_{\bf X}({\bf R})^*C_{\bf X}({\bf R}),\qquad {\bf X}\in {\bf B_n}(\cH),
 $$
where $C_{\bf X}({\bf R})$ is given in the theorem.
Since the  ${\bf X}\mapsto \boldsymbol{\cP}({\bf R}, {\bf X})$ is continuous on ${\bf B_n}(\cH)$ in the operator norm topology, and taking $r\to 1$ in relation  \eqref{CB}, we obtain that
$$
\frac{1}{c^2}C_{\bf B}({\bf R})^*C_{\bf B}({\bf R})\leq  C_{\bf A}({\bf R})^*C_{\bf A}({\bf R})\leq c^2 C_{\bf B}({\bf R})^*C_{\bf B}({\bf R})
$$
which, due to the fact that $C_{\bf A}({\bf R})$ and $C_{\bf B}({\bf R})$ are invertible operators, implies
$$
(C_{\bf A}({\bf R})^{-1})^*C_{\bf B}({\bf R})^*C_{\bf B}({\bf R})C_{\bf A}({\bf R})^{-1}\leq c^2I
$$
and
$$
(C_{\bf B}({\bf R})^{-1})^*C_{\bf A}({\bf R})^*C_{\bf A}({\bf R})C_{\bf B}({\bf R})^{-1}\leq c^2I.
$$
Consequently,
$$
t:=\max \left\{\left\|C_{\bf A}({\bf R})C_{\bf B}({\bf R})^{-1}\right\|, \left\|C_{\bf B}({\bf R})C_{\bf A}({\bf R})^{-1}\right\|\right\}\leq c
$$
which implies  $\ln t\leq \delta_\cP({\bf A}, {\bf B})$.
To prove the reverse inequality, note that
$$
\left\|C_{\bf A}({\bf R})C_{\bf B}({\bf R})^{-1}\right\|\leq t \quad \text{and} \quad \left\|C_{\bf B}({\bf R})C_{\bf A}({\bf R})^{-1}\right\|\leq t.
$$
These inequalities imply
 $$
\frac{1}{t^2}C_{\bf B}({\bf R})^*C_{\bf B}({\bf R})\leq  C_{\bf A}({\bf R})^*C_{\bf A}({\bf R})\leq t^2 C_{\bf B}({\bf R})^*C_{\bf B}({\bf R})
$$
which is equivalent to
$$
\frac{1}{t^2}\boldsymbol{\cP}({\bf R}, {\bf B})\leq  \boldsymbol{\cP}({\bf R}, {\bf A})\leq t^2 \boldsymbol{\cP}({\bf R}, {\bf B})
$$
and shows that $t\geq 1$. Hence,
$
\frac{1}{t^2}\boldsymbol{\cP}({\bf S}, {\bf B})\leq  \boldsymbol{\cP}({\bf R}, {\bf A})\leq t^2 \boldsymbol{\cP}({\bf S}, {\bf B}).
$
Applying the Berezin transform  at $r{\bf R}$, $r\in [0,1)$, and using the fact that ${\boldsymbol\cB_{r\bf R}}\otimes id$ is a unital completely positive map, we deduce that ${\bf A}\,\overset{P}{{\underset{t}\sim}}\, {\bf B}$.
Consequently,  $ \delta_\cP({\bf A}, {\bf B})\leq \ln t$, which completes the proof of the theorem.
\end{proof}

We remark that if ${\bf A}$ and ${\bf B}$ are   in
${\bf B}_{\bf n}(\cH)^-$, then ${\bf A}\,\overset{P}{\sim}\, {\bf B}$ if and only if $r{\bf A}\,\overset{P}{\sim}\, r{\bf B}$ for any $r\in [0,1)$ and
$\sup_{r\in [0,1)} \omega_\cP(r{\bf A}, r{\bf B})<\infty$. In this case, Theorem \ref{dp} implies
$$
\delta_{\cP}({\bf A}, {\bf B})=\ln \max \left\{\sup_{r\in [0,1)}\left\|C_{\bf A}({\bf R}) C_{\bf B}({\bf R})^{-1}\right\|, \sup_{r\in [0,1)}\left\|C_{\bf B}({\bf R})C_{\bf A}({\bf R})^{-1}\right\|\right\}.
$$

According to Lemma \ref{les}, if ${\bf X}=(X_1, \ldots, X_k)\in {\bf B}_{\bf n}(\cH)^-$,
 then
${\bf X}\overset{P}{\prec}\, {0}$ if and only if  the joint  spectral radius $r(X_i)<1$ for any $i\in \{1,\ldots, k\}$.
Set
$$
{\bf B}_{\bf n}(\cH)^-_0:=\left\{ {\bf X}\in {\bf B}_{\bf n}(\cH)^-:\ {\bf X}\overset{P}{\prec}\, {0}\right\}
$$
and note that Theorem \ref{foias2} implies that ${\bf B}_{\bf n}(\cH)\subset {\bf B}_{\bf n}(\cH)^-_0$. We define the map
$d_\cP$ on ${\bf B}_{\bf n}(\cH)^-_0\times {\bf B}_{\bf n}(\cH)^-_0$  by setting
$$
d_\cP({\bf A}, {\bf B}):=\sup_{r\in [0,1)}\left\|\boldsymbol{\cP}({\bf R}, r{\bf A})-\boldsymbol{\cP}({\bf R}, r{\bf B})\right\|, \qquad {\bf A}, {\bf B}\in {\bf B}_{\bf n}(\cH)^-_0.
$$
Note that if ${\bf X}\in {\bf B}_{\bf n}(\cH)^-_0$, then ${\bf X}\overset{P}{\prec}\, {0}$, which shows that there is $c\geq 1$ such that
$\boldsymbol{\cP}({\bf R}, r{\bf A})\leq c^2 I$ for any $r\in [0,1)$. Consequently,
$d_\cP({\bf A}, {\bf B})<\infty$ for any $ {\bf A}, {\bf B}\in {\bf B}_{\bf n}(\cH)^-_0$. If $d_\cP({\bf A}, {\bf B})=0$, then $d_\cP({\bf A}, {\bf B})=0$ and,
as in the proof of Lemma \ref{beta}, we deduce that ${\bf A}={\bf B}$.
Now, it is clear that $d_\cP$ is a metric on ${\bf B}_{\bf n}(\cH)^-_0$.

\begin{theorem}\label{prop-dp}
The map $d_\cP$ has the folowing properties:
\begin{enumerate}
\item[(i)] $d_\cP$ is a complete metric on $  {\bf B}_{\bf n}(\cH)^-_0$;
\item[(ii)]
the $d_\cP$-topology is stronger than the norm topology on $ {\bf B}_{\bf n}(\cH)^-_0$;
\item[(iii)]
the  $d_\cP$-topology coincides with the norm topology on ${\bf B}_{\bf n}(\cH)$.

\end{enumerate}

\end{theorem}
\begin{proof}
Let ${\bf A}=(A_1,\ldots, A_k)$ and ${\bf B}=(B_1,\ldots, B_k)$ be in ${\bf B}_{\bf n}(\cH)^-_0$.
Set $w:=\sum_{\alpha\in \FF_{n_i}^+} e_\alpha^i\otimes h_\alpha\in F^2(H_{n_i})\otimes \cH$, where $h_\alpha\in \cH$ and $\sum_{\alpha\in \FF_{n_i}^+} \|h_\alpha\|^2<\infty$, and let
$$
\tilde{w}:=\sum_{\alpha\in \FF_{n_i}^+} (\underbrace{1\otimes\cdots \otimes  1}_{\text{${i-1}$
times}}  \otimes\,  e_\alpha^i\otimes 1\otimes \cdots \otimes 1)\otimes h_\alpha
$$
be in $F^2(H_{n_1})\otimes \cdots \otimes F^2(H_{n_k})$. Note that
\begin{equation}\label{PP}
\left< P(R_i, rA_i)w,w\right>=\left<\boldsymbol{\cP}({\bf R}, r{\bf A})\tilde{w}, \tilde{w}\right>
\end{equation}
for any $w\in F^2(H_{n_i})\otimes \cH$, where $P(R_i, rA_i)$ is the Poisson kernel associate with the row contraction $rA_i$ and where $R_i=(R_{i,1},\ldots, R_{i,n_i})$ is the $n_i$-tuple of right creation operators acting on the full Fock space $F^2(H_{n_i})$.

Setting $\Lambda_{A_q}:={\bf R}_{q,1}^*\otimes A_{q,1}+\cdots +{\bf R}_{q,n_q}^*\otimes A_{q,n_q}$, $q\in \{1,\ldots, k\}$,  we deduce that
$$
r\Lambda_{A_q}=\frac{1}{2\pi}\int_0^{2\pi} e^{-it}P(re^{it} R_q,A_q),\qquad r\in [0,1),
$$
and, consequently,
\begin{equation*}
\begin{split}
\|rA_q-rB_q\|&=\|r\Lambda_{A_q}-r\Lambda_{B_q}\|\\
&=\left\|
\frac{1}{2\pi}\int_0^{2\pi} e^{-it}\left[P(re^{it} R_q,A_q)-P(re^{it} R_q,B_q)\right]dt\right\|\\
&\leq \sup_{r\in [0,1)}\left\|P(re^{it} R_q,A_q)-P(re^{it} R_q,B_q)\right\|\\
&\leq
\left\|P( R_q,rA_q)-P(R_q,rB_q)\right\|
\end{split}
\end{equation*}
for any $r\in [0,1)$. Hence and using relation \eqref{PP}, we deduce that
\begin{equation} \label{aqbq}
\|A_q-B_q\|\leq d_\cP({\bf A}, {\bf B}),
\qquad {\bf A}, {\bf B}\in {\bf B}_{\bf n}(\cH)^-_0.
\end{equation}
Let ${\bf A}^{(m)}:=(A_1^{(m)},\ldots, A_k^{(m)})$ be a $d_\cP$-Cauchy sequence in ${\bf B}_{\bf n}(\cH)^-_0$. Due to the latter inequality, for each $i\in \{1,\ldots, k\}$, $\{A_i^{(m)}\}_{m=1}^\infty$ is a Cauchy sequence in the norm topology of $[B(\cH)^{n_i}]_1^-$. Consequently, there exists $T_i\in [B(\cH)^{n_i}]_1^-$ such that $\|A_i^{(m)}-T_i\|\to 0$, as $m\to \infty$.
Since ${\bf A}^{(m)}\in {\bf B}_{\bf n}(\cH)^-$, so is ${\bf T}=(T_1,\ldots, T_k)$.
Since ${\bf A}^{(m)}$ is a $d_\cP$-Cauchy sequence, there is $m_0\in \NN$ such that $d_\cP({\bf A}^{(m)}, {\bf A}^{(m_0)})\leq 1$ for any $m\geq m_0$.
On the other hand, ${\bf A}^{(m_0)}\in {\bf B}_{\bf n}(\cH)^-_0$ which shows that
${\bf A}^{(m_0)}\overset{P}{\prec}\, {0}$ and, consequently, there is $c\geq 1$ such that
$\boldsymbol{\cP}({\bf R}, r{\bf A}^{(m_0)})\leq c^2I$ for any $r\in [0,1)$. Hence, we deduce that
$$
\boldsymbol{\cP}({\bf R}, r{\bf A}^{(m)})\leq \left(d_\cP({\bf A}^{(m)}, {\bf A}^{(m_0)})+\left\|\boldsymbol{\cP}({\bf R}, r{\bf A}^{(m_0)})\right\|\right)I
\leq (c^2+1)I
$$
for any $m\geq m_0$. Using the continuity of the map ${\bf X}\mapsto \boldsymbol{\cP}({\bf R}, {\bf X})$ on ${\bf B}_{\bf n}(\cH)$ and taking $m\to \infty$, we obtain
$\boldsymbol{\cP}({\bf R}, r{\bf T})\leq (c^2+1)I$ for any $r\in [0,1)$, which shows that ${\bf T}\overset{P}{\prec}\, {0}$. Thus ${\bf T}\in {\bf B}_{\bf n}(\cH)^-_0$.

Set $G_r:=\boldsymbol{\cP}({\bf S}, r{\bf A}^{(m)})-\boldsymbol{\cP}({\bf S}, r{\bf T})$ and let $0\leq r_1<r_2<1$. Using the noncommutative Berezin transform
we have $\left(\boldsymbol{\cB}_{\frac{r_1}{r_2}{\bf R}}\otimes id\right)[G_{r_2}]=G_{r_1}$. Hence, $\|G_{r_1}\|\leq \|G_{r_2}\|$.  Using this result and the fact that
$$
 \boldsymbol{\cP}({\bf R}, {\bf X})=C_{\bf X}({\bf R})^*C_{\bf X}({\bf R}),\qquad {\bf X}\in {\bf B_n}(\cH),
 $$
where $$C_{\bf X}({\bf R}):=(I\otimes \boldsymbol\Delta_{\bf X}(I)^{1/2})\prod_{i=1}^k(I-{\bf R}_{i,1}\otimes
X_{i,1}^*-\cdots -{\bf R}_{i,n_i}\otimes X_{i,n_i}^*)^{-1}$$ for any
${\bf X}=(X_1,\ldots, X_k)\in {\bf B_n}(\cH)$  with $X_i=(X_{i,1},\ldots, X_{i, n_i})$, we obtain
\begin{equation*}
\begin{split}
d_\cP({\bf A}^{(m)}, {\bf T})&= \left\|\lim_{r\to 1}\left[\boldsymbol{\cP}({\bf R}, r{\bf A})-\boldsymbol{\cP}({\bf R}, r{\bf B})\right]\right\|\\
&=\left\| \lim_{r\to 1} \left[
 C_{r{\bf A}^{(m)}}({\bf R})^*C_{r{\bf A}^{(m)}}({\bf R})-  C_{r{\bf T}}({\bf R})^*C_{r{\bf T}}({\bf R})\right]\right\|\\
 &=\left\|
 C_{{\bf A}^{(m)}}({\bf R})^*C_{{\bf A}^{(m)}}({\bf R})-  C_{\bf T}({\bf R})^*C_{\bf T}({\bf R})
 \right\|.
\end{split}
\end{equation*}
The latter equality holds due to the fact that the spectral radius is upper semicontinuous and, for each $q\in \{1,\ldots, k\}$,  the spectral radii  of
$\Lambda_{{A_q}^{(m)}}$ and  $\Lambda_{T_q}$ are strictly less than $1$. Indeed, in this case we have
$C_{r{\bf A}^{(m)}}({\bf R})\to C_{{\bf A}^{(m)}}({\bf R})$ and  $C_{r{\bf T}}({\bf R})\to C_{{\bf T}}({\bf R})$ in the norm topology,
as $r\to 1$. On the other hand, since that map $Y\mapsto Y^{-1}$ is norm continuous on the open set of invertible operators, $C_{{\bf A}^{(m)}}({\bf R})\to C_{\bf T}({\bf R})$ in norm as $m\to \infty$. Consequently, $d_\cP({\bf A}^{(m)}, {\bf T})\to 0$ as $m\to \infty$, which completes the proof of part (i).
Due to relation \eqref{aqbq}, part (ii) is clear.

To prove part (iii), assume that  ${\bf A}, {\bf B}\in {\bf B}_{\bf n}(\cH)$.
Note that, as above, we have
\begin{equation*}
\begin{split}
d_\cP({\bf A}, {\bf B})&= \left\|\lim_{r\to 1}\left[\boldsymbol{\cP}({\bf R}, r{\bf A})-\boldsymbol{\cP}({\bf R}, r{\bf B})\right]\right\|\\
 &=\left\|
 C_{{\bf A}}({\bf R})^*C_{{\bf A}}({\bf R})-  C_{\bf T}({\bf R})^*C_{\bf T}({\bf R})
 \right\|.
\end{split}
\end{equation*}
According to Proposition 4.1 from \cite{Po-pluriharmonic-polyball}, the map ${\bf X}\to  C_{{\bf X}}({\bf R})$  is continuous in the operator norm topology. Now, one can easily complete the proof of part (iii).
\end{proof}

\begin{lemma}\label{dp2}If ${\bf A}, {\bf B}\in {\bf B_n}(\cH)^-$  are such that
${\bf A}\,\overset{P}{\sim}\, {\bf B}$, then
\begin{equation*}
 \delta_\cP({\bf A},{\bf B})= \frac{1}{2}\sup\left|\ln\frac{\left<\boldsymbol{\cP}({\bf R}, r{\bf A})x,x\right>}{\left<\boldsymbol{\cP}({\bf R}, r{\bf B})x,x\right>}\right|,
 \end{equation*}
where the supremum is taken over all $x\in F^2(H_{n_1})\otimes \cdots \otimes F^2(H_{n_k})\otimes \cH$, $x\neq 0$, and all $r\in [0,1)$.
\end{lemma}
\begin{proof}
Let ${\bf A}, {\bf B}\in {\bf B_n}(\cH)^-$  be  such that
${\bf A}\,\overset{P}{{\underset{c}\sim}}\, {\bf B}$ with $c\geq 1$.
Then
\begin{equation*}
\frac{1}{c^2}\left<\boldsymbol{\cP}({\bf R}, r{\bf B})x,x\right>\leq  \left<\boldsymbol{\cP}({\bf R}, r{\bf A})x,x\right>\leq c^2 \left<\boldsymbol{\cP}({\bf R}, r{\bf B})x,x\right>
\end{equation*}
 for any $r\in [0,1)$ and any  $x\in F^2(H_{n_1})\otimes \cdots \otimes F^2(H_{n_k})\otimes \cH$ with $x\neq 0$. Since
$\boldsymbol{\cP}({\bf R}, r{\bf B})$ is an invertible operator, we deduce that
$$
-\ln c\leq \frac{1}{2} \ln\frac{\left<\boldsymbol{\cP}({\bf R}, r{\bf A})x,x\right>}{\left<\boldsymbol{\cP}({\bf R}, r{\bf B})x,x\right>}\leq \ln c
$$
which implies
\begin{equation*} M:=
\frac{1}{2}\sup\left|\ln\frac{\left<\boldsymbol{\cP}({\bf R}, r{\bf A})x,x\right>}{\left<\boldsymbol{\cP}({\bf R}, r{\bf B})x,x\right>}\right|\leq  \delta_\cP({\bf A},{\bf B}).
 \end{equation*}
 To prove the reverse inequality, note that
 $$
 \frac{1}{2} \left|\ln\frac{\left<\boldsymbol{\cP}({\bf R}, r{\bf A})x,x\right>}{\left<\boldsymbol{\cP}({\bf R}, r{\bf B})x,x\right>}\right|\leq M
 $$
which is equivalent to
\begin{equation*}
e^{-2M}\left<\boldsymbol{\cP}({\bf R}, r{\bf B})x,x\right>\leq  \left<\boldsymbol{\cP}({\bf R}, r{\bf A})x,x\right>\leq e^{2M} \left<\boldsymbol{\cP}({\bf R}, r{\bf B})x,x\right>
\end{equation*}
 for any $r\in [0,1)$ and any  $x\in F^2(H_{n_1})\otimes \cdots \otimes F^2(H_{n_k})\otimes \cH$ with $x\neq 0$.
 Consequently, $\delta_\cP({\bf A},{\bf B})\leq m$. The proof is complete.
\end{proof}

\begin{theorem}\label{dp-dh} Let  $\Delta$ be a Poisson part of  ${\bf B_n}(\cH)_0^-$.
Then   the following properties hold:
\begin{enumerate}
\item[(i)]  $\delta_\cP$ is a complete metric on $\Delta$.
\item[(ii)] the $\delta_\cP$-topology is stronger than the $d_\cP$-topology on $\Delta$.
\item[(iii)] the $\delta_\cP$-topology,  the $d_\cP$-topology, and the operator norm topology coincide on the open polyball ${\bf B}_{\bf n}(\cH)$.
    \item[(iv)] the $\delta_H$-topology is stronger that the $\delta_\cP$-topology on ${\bf B}_{\bf n}(\cH)$
\end{enumerate}
\end{theorem}

\begin{proof}
Let ${\bf A}, {\bf B}\in \Delta$. Due to the definition of $\omega_\cP$, we have
$$
\boldsymbol{\cP}({\bf R}, r{\bf A})\leq\omega_\cP({\bf A}, {\bf B})^2 \boldsymbol{\cP}({\bf R}, r{\bf B}), \qquad r\in [0,1),
$$
which implies
$$
\boldsymbol{\cP}({\bf R}, r{\bf A})-\boldsymbol{\cP}({\bf R}, r{\bf B})\leq[\omega_\cP({\bf A}, {\bf B})^2-1]M_{\bf B} I, \qquad r\in [0,1),
$$
where $M_{\bf B}:=\sup_{r\in [0,1)} \|\boldsymbol{\cP}({\bf R}, r{\bf B})\|<\infty$ due to the fact that ${\bf B}\overset{P}{\prec}\, {0}$. Similarly, we can obtain the inequality
$$
\boldsymbol{\cP}({\bf R}, r{\bf B})-\boldsymbol{\cP}({\bf R}, r{\bf A})\leq[\omega_\cP({\bf A}, {\bf B})^2-1]M_{\bf A} I, \qquad r\in [0,1).
$$
Consequently,  since $\boldsymbol{\cP}({\bf R}, r{\bf B})-\boldsymbol{\cP}({\bf R}, r{\bf A})$ is a self-adjoint operator, we obtain
$$
\left\|\boldsymbol{\cP}({\bf R}, r{\bf B})-\boldsymbol{\cP}({\bf R}, r{\bf A})\right\|\leq \max\{M_{\bf A}, M_{\bf B}\}[\omega_\cP({\bf A}, {\bf B})^2-1],\qquad r\in [0,1).
$$
Hence, we deduce that
\begin{equation}\label{dpm}
d_\cP({\bf A}, {\bf B})\leq \max\{M_{\bf A}, M_{\bf B}\}\left(e^{2\delta_\cP({\bf A}, {\bf B})}-1\right).
\end{equation}
Now, we prove that $\delta_\cP$ is a complete metric on $\Delta$.
Let $\{{\bf A}^{(m)}\}_{m=1}^\infty$ be a $\delta_\cP$-Cauchy sequence in $\Delta$.
For any $\epsilon>0$, there is $m_0\in \NN$ such that
\begin{equation}
\label{Cau}
\delta_\cP({\bf A}^{(m)}, {\bf A}^{(p)})<\epsilon, \qquad m,p\geq m_0.
\end{equation}
Since ${\bf A}^{(m)}\overset{P}{\prec}\, {\bf A}^{(m_0)}$ and ${\bf A}^{(m)}\overset{P}{\prec}\, 0$, we have
\begin{equation}\label{po}
\begin{split}
\boldsymbol{\cP}({\bf R}, r{\bf A}^{(m)})&\leq \omega_\cP({\bf A}^{(m)}, {\bf A}^{(m_0)})\boldsymbol{\cP}({\bf R}, r{\bf A}^{(m_0)})\\
&\leq \omega_\cP({\bf A}^{(m)}, {\bf A}^{(m_0)}) \sup_{r\in [0,1)}\|\boldsymbol{\cP}({\bf R}, r{\bf A}^{(m_0)})\| I
\end{split}
\end{equation}
for any $r\in[0,1)$. Hence and using \eqref{Cau}, we obtain
$$
\left\|\boldsymbol{\cP}({\bf R}, r{\bf A}^{(m)})\right\|\leq
 \sup_{r\in [0,1)}\|\boldsymbol{\cP}({\bf R}, r{\bf A}^{(m_0)})\| e^{2\epsilon}
 $$
 for any $m\geq m_0$ and $r\in [0,1)$. Since $\sup_{r\in [0,1)}\|\boldsymbol{\cP}({\bf R}, r{\bf A}^{(m_0)})\|<\infty$, we deduce that
 the sequence
 $\left\{\sup_{r\in [0,1)}\left\|\boldsymbol{\cP}({\bf R}, r{\bf A}^{(m)})\right\|\right\}_{m=1}^\infty$ is bounded. The inequality
 \eqref{dpm}, implies that $\{{\bf A}^{(m)}\}_{m=1}^\infty$ is a $d_\cP$-Cauchy sequence.
 According to  Theorem \ref{prop-dp}, there exists ${\bf A}\in {\bf B}_{\bf n}(\cH)^-_0$ such that
 \begin{equation}
 \label{daa}
 \lim_{m\to \infty}d_\cP({\bf A}^{(m)}, {\bf A})\to 0.
 \end{equation}
Combining relations \eqref{Cau}  and \eqref{po}
\begin{equation}\label{po2}
\begin{split}
\boldsymbol{\cP}({\bf R}, r{\bf A}^{(m)})&\leq \omega_\cP({\bf A}^{(m)}, {\bf A}^{(m_0)})\boldsymbol{\cP}({\bf R}, r{\bf A}^{(m_0)})\\
&\leq  e^{2\epsilon} \sup_{r\in [0,1)}\boldsymbol{\cP}({\bf R}, r{\bf A}^{(m_0)})
\end{split}
\end{equation}
for any $m\geq m_0$ and $r\in [0,1)$. Using relation \eqref{daa} and passing to the limit as $m\to \infty$ in  relation \eqref{po2}, we obtain
\begin{equation}
\label{pe}
\boldsymbol{\cP}({\bf R}, r{\bf A})\leq e^{2\epsilon} \boldsymbol{\cP}({\bf R}, r{\bf A}^{(m_0)})
\end{equation}
for any $r\in [0,1)$, which shows that ${\bf A}\overset{P}{\prec}\, {\bf A}^{(m_0)}$. On the other hand, since ${\bf A}^{(m_0)}\overset{P}{\prec}\, {\bf A}^{(m)}$  for any $m\geq m_0$, relation  \eqref{Cau} implies
\begin{equation}\label{po3}
\begin{split}
\boldsymbol{\cP}({\bf R}, r{\bf A}^{(m_0)})&\leq \omega_\cP({\bf A}^{(m_0)}, {\bf A}^{(m)})\boldsymbol{\cP}({\bf R}, r{\bf A}^{(m)})\\
&\leq  e^{2\epsilon} \boldsymbol{\cP}({\bf R}, r{\bf A}^{(m)})
\end{split}
\end{equation}
for any $m\geq m_0$ and $r\in [0,1)$.
Due to Theorem \ref{prop-dp}, the $d_\cP$-topology is stronger than the norm topology on ${\bf B_n}(\cH)_0^-$. Therefore, relation \eqref{daa} implies ${\bf A}^{(m)}\to {\bf A}\in {\bf B_n}(\cH)_0^-$ in the operator norm topology.
Taking $m\to \infty$ in  relation \eqref{po3}, we obtain
\begin{equation}\label{po4}
\begin{split}
\boldsymbol{\cP}({\bf R}, r{\bf A}^{(m_0)})\leq
  e^{2\epsilon} \boldsymbol{\cP}({\bf R}, r{\bf A})
\end{split}
\end{equation}
for any  $r\in [0,1)$. Consequently, ${\bf A}^{(m_0)}\overset{P}{\prec}\, {\bf A}$ which together with relation  \eqref{po3} imply
${\bf A}^{(m_0)}\,\overset{P}{{\sim}}\, {\bf A}$. Thus ${\bf A}\in \Delta$.
Note that the inequalities \eqref{pe} and \eqref{po4} show that $\omega_\cP({\bf A}^{(m_0)}, {\bf A})\leq e^{2\epsilon}$ and, therefore,
$\delta_\cP({\bf A}^{(m_0)}, {\bf A})\leq \epsilon$. Now, using relation \eqref{Cau}, we obtain
$\delta_\cP({\bf A}^{(m)}, {\bf A})\leq 2\epsilon$ for any $m\geq m_0$. This shows that
$\delta_\cP({\bf A}^{(m)}, {\bf A})\to 0$ as $m\to \infty$ and completes the proof that $\delta_\cP$  is a complete metric on $\Delta$. Note that we have already proved part (ii) of the theorem.

Now, we prove part (iii). Assume that  ${\bf A}, {\bf B}\in {\bf B_n}(\cH)$.
Since $\boldsymbol{\cP}({\bf R}, {\bf B})$ is a positive invertible operator, we have
$I\leq \|\boldsymbol{\cP}({\bf R}, {\bf B})^{-1}\| \boldsymbol{\cP}({\bf R}, {\bf B})$. Using the fact that the Berezin transform $ \boldsymbol{\cB}_{r{\bf R}}\otimes id$ is a completely positive linear map, we deduce that
$I\leq \|\boldsymbol{\cP}({\bf R}, {\bf B})^{-1}\| \boldsymbol{\cP}({\bf R}, r{\bf B})$ for any $r\in [0,1)$. Consequently,
\begin{equation*}
\begin{split}
\frac{\left<\boldsymbol{\cP}({\bf R}, r{\bf A})x,x\right>}{\left<\boldsymbol{\cP}({\bf R}, r{\bf B})x,x\right>}
-1&\leq \frac{\|\boldsymbol{\cP}({\bf R}, {\bf B})^{-1}\|}{\|x\|}
\left<(\boldsymbol{\cP}({\bf R}, r{\bf A})-\boldsymbol{\cP}({\bf R}, r{\bf B}))x,x\right>\\
&\leq \|\boldsymbol{\cP}({\bf R}, {\bf B})^{-1}\|d_\cP({\bf A}, {\bf B})
\end{split}
\end{equation*}
for any  $x\in F^2(H_{n_1})\otimes \cdots \otimes F^2(H_{n_k})\otimes \cH$, $x\neq 0$, and all $r\in [0,1)$. Hence, we obtain
$$
\ln\frac{\left<\boldsymbol{\cP}({\bf R}, r{\bf A})x,x\right>}{\left<\boldsymbol{\cP}({\bf R}, r{\bf B})x,x\right>}
\leq \ln \left(1+ \|\boldsymbol{\cP}({\bf R}, {\bf B})^{-1}\|d_\cP({\bf A}, {\bf B})\right).
$$
Interchanging ${\bf A}$  with ${\bf B}$, we obtain a similar inequality. Putting the two inequalities together, we deduce that
$$
\left|\ln\frac{\left<\boldsymbol{\cP}({\bf R}, r{\bf A})x,x\right>}{\left<\boldsymbol{\cP}({\bf R}, r{\bf B})x,x\right>}\right|
\leq \ln \left(1+ \max\left\{\|\boldsymbol{\cP}({\bf R}, {\bf A})^{-1}\|,\|\boldsymbol{\cP}({\bf R}, {\bf B})^{-1}\|\right\}d_\cP({\bf A}, {\bf B})\right)
$$
for any  $x\in F^2(H_{n_1})\otimes \cdots \otimes F^2(H_{n_k})\otimes \cH$, $x\neq 0$, and all $r\in [0,1)$.
Using Lemma \ref{dp2}, we obtain
\begin{equation}
\label{ineq-d}
 \delta_\cP({\bf A},{\bf B})
 \leq \frac{1}{2}
 \ln \left(1+ \max\left\{\|\boldsymbol{\cP}({\bf R}, {\bf A})^{-1}\|,\|\boldsymbol{\cP}({\bf R}, {\bf B})^{-1}\|\right\}d_\cP({\bf A}, {\bf B})\right).
 \end{equation}
Let $\{{\bf A}^{(m)}\}_{m=1}^\infty$ be a sequence of elements in ${\bf B}_{\bf n}(\cH)$ and ${\bf A}\in {\bf B}_{\bf n}(\cH)$ be such that $d_\cP({\bf A}^{(m)},{\bf A} )\to 0$, as $m\to\infty$.
This implies
$\boldsymbol{\cP}({\bf R}, {\bf A}^{(m)})\to \boldsymbol{\cP}({\bf R}, {\bf A})$ in the operator norm topology, as $m\to\infty$. Since ${\bf A}^{(m)}, {\bf A}\in {\bf B}_{\bf n}(\cH)$, the operators
$\boldsymbol{\cP}({\bf R}, {\bf A}^{(m)})$ and $\boldsymbol{\cP}({\bf R}, {\bf A})$ are invertible and, consequently,
$\boldsymbol{\cP}({\bf R}, {\bf A}^{(m)})^{-1}\to \boldsymbol{\cP}({\bf R}, {\bf A})^{-1}$ in the operator norm topology, as $m\to\infty$. Now, it is clear that  there is $M>0$ such that
$\left\|\boldsymbol{\cP}({\bf R}, {\bf A}^{(m)})^{-1}\right\|\leq M$
for any $m\in \NN$. Applying inequality \eqref{ineq-d}, we obtain
$$
 \delta_\cP({\bf A}^{(m)},{\bf A})\leq \frac{1}{2} \ln\left[1+Md_\cP({\bf A}^{(m)},{\bf A})\right],\qquad m\in \NN.
 $$
Since $d_\cP({\bf A}^{(m)},{\bf A})\to$ as $m\to \infty$, we deduce that
$\delta_\cP({\bf A}^{(m)},{\bf A})\to 0$ as well. This shows that the $d_\cP$-topology on ${\bf B}_{\bf n}(\cH)$ is stronger than the $\delta_\cP$-topology. On the other hand, due to part (ii) of this theorem and the fact that ${\bf B}_{\bf n}(\cH)$ is a Poisson part in ${\bf B}_{\bf n}(\cH)^-_0$, we conclude that the $\delta_\cP$-topology coincides with the
$d_\cP$-topology on ${\bf B}_{\bf n}(\cH)$.
Now, using Theorem \ref{prop-dp} part (iii), we complete the proof of item (iii).

According to Theorem \ref{foias2}, the open unit polyball ${\bf B}_{\bf n}(\cH)$ is the  Harnack (respectively, Poisson) part  of ${\bf B}_{\bf n}(\cH)^-$ which contains the origin. Since $\delta_\cP({\bf X}, {\bf Y})\leq \delta_H({\bf X}, {\bf Y})$ for ${\bf X},{\bf Y}\in {\bf B}_{\bf n}(\cH)$, item (iv) follows.
The proof is complete.
\end{proof}

\begin{corollary} If  $\Delta$ is a Poisson part of  ${\bf B_n}(\cH)_0^-$, then
 \begin{equation*}
 \delta_\cP({\bf A},{\bf B})
 \geq \frac{1}{2}
 \ln \left(1+ \frac{d_\cP({\bf A},{\bf B})}{\max\left\{
 \sup_{r\in [0,1)}\|\boldsymbol{\cP}({\bf R}, {\bf A}) \|, \sup_{r\in [0,1)}\|\boldsymbol{\cP}({\bf R}, {\bf B})\|\right\}} \right).
 \end{equation*}

If ${\bf A},{\bf B}\in {\bf B}_{\bf n}(\cH)$, then
 \begin{equation*}
 \delta_\cP({\bf A},{\bf B})
 \leq \frac{1}{2}
 \ln \left(1+ \max\left\{\|\boldsymbol{\cP}({\bf R}, {\bf A})^{-1}\|,\|\boldsymbol{\cP}({\bf R}, {\bf B})^{-1}\|\right\}d_\cP({\bf A}, {\bf B})\right).
 \end{equation*}
\end{corollary}

\bigskip

\section{Hyperbolic metric on the regular  polydisk}

Using a characterization of positive free $k$-pluriharmonic functions on regular polydisks   and the results of the previous sections, we prove that the Harnack parts and the Poisson parts on  $ {\bf D}^k(\cH)^-$ coincide, and so are   the metrics $\delta_H$ and $\delta_\cP$.   We show that the hyperbolic metric $\delta_H$ on ${\bf D}^k(\cH)$ has similar properties to the Poincar\' e distance on the open unit disc $\DD$.

Let $\Omega\subset {\bf F}_{\bf n}^+\times {\bf F}_{\bf n}^+$ be the set of all  pairs  $(\boldsymbol \alpha, \boldsymbol \beta)$ where $\boldsymbol \alpha =(\alpha_1,\ldots, \alpha_k), \boldsymbol \beta=(\beta_1,\ldots, \beta_k)$ are in ${\bf F }_{\bf n}^+:=\FF_{n_1}^+\otimes \cdots \otimes \FF_{n_k}^+$ such that
$\alpha_i,\beta_i\in \FF_{n_i}^+, |\alpha_i|=m_i^-$, and $ |\beta_i|=m_i^+$ for some $m_i\in \ZZ$.
    In \cite{Po-pluriharmonic-polyball}, we proved that a map $F:{\bf B_n}(\cH)\to B(\cE)\otimes_{min} B(\cH)$, with $F(0)=I$,   is a positive free $k$-pluriharmonic function on the regular  polyball if and only if   it has the form
 $$
 F({\bf X})=
 \sum_{(\boldsymbol \alpha, \boldsymbol \beta)\in \Omega}
 P_\cE{\bf V}_{\widetilde{\boldsymbol\alpha}}^*{\bf V}_{\widetilde{\boldsymbol\beta}}|_\cE\otimes
 {\bf X}_{\boldsymbol \alpha} {\bf X}_{\boldsymbol\beta}^*,
 $$
 where ${\bf V}=(V_1,\ldots, V_k)$ is a
 $k$-tuple   of commuting row isometries on a space $\cK\supset \cE$ such that
  $$
  \sum_{(\boldsymbol \alpha, \boldsymbol \beta)\in \Omega}
 P_\cE{\bf V}_{\widetilde{\boldsymbol\alpha}}^*{\bf V}_{\widetilde{\boldsymbol\beta}}|_\cE\otimes r^{|\boldsymbol\alpha|+|\boldsymbol\beta|}
 {\bf S}_{\boldsymbol \alpha} {\bf S}_{\boldsymbol\beta}^*\geq 0,\qquad r\in [0,1),
  $$
the series is convergent in the operator topology,
and
 $\widetilde{\boldsymbol\alpha}=(\widetilde{\alpha}_1,\ldots \widetilde{\alpha}_k)$ is the reverse of ${\boldsymbol\alpha}=({\alpha_1},\ldots {\alpha_k})$, i.e.
  $\widetilde \alpha_i= g^i_{i_k}\cdots g^i_{i_1}$ if
   $\alpha_i=g^i_{i_1}\cdots g^i_{i_k}\in\FF_{n_i}^+$.
As a consequence of this result we have the following characterization of positive
 free $k$-pluriharmonic function  on regular polidisks. We include a proof for completeness.
In what follows we consider the {\it regular polydisc}   ${\bf D}^k(\cH):={\bf B}_{(1,\ldots, 1)}(\cH)$.

\begin{proposition} \label{polydisk}
 Let  $F:{\bf D}^k(\cH)\to B(\cE)\otimes_{min} B(\cH)$ be a free $k$-pluriharmonic function with $F(0)=I$. Then $F$ is positive if and only if
$$
F({\bf X})=(P_\cE\otimes I)\boldsymbol{\cP}({\bf U}, {\bf X})|_{\cE\otimes \cH},
$$
where  ${\bf U}=(U_1,\ldots, U_k)$ is a $k$-tuple of  commuting unitaries on a Hilbert space $\cK\supset \cE$, and the free pluriharmonic Poisson kernel  $\boldsymbol{\cP}({\bf U}, {\bf X})$   is equal to
\begin{equation*}
  \sum_{m_1\in \ZZ}\cdots \sum_{m_k\in \ZZ} ({U}_{1}^*)^{m_1^-}\cdots ({ U}_{k}^*)^{m_k^-}{U}_{1}^{m_1^+}\cdots ({ U}_{k})^{m_k^+}
\otimes { X}_{1}^{m_1^-}\cdots {X}_{k}^{m_k^-}({X}_{1}^*)^{m_1^+}\cdots ({X}_{k}^*)^{m_k^+}
\end{equation*}
for any ${\bf X}=(X_1,\ldots, X_k)\in {\bf D}^k(\cH)$,
where the convergence of the multi-series is in the operator norm topology.
\end{proposition}
\begin{proof}

Assume that $F$ is a positive free $k$-pluriharmonic function with $F(0)=I$.
According to the above-mentioned result, there is a $k$-tuple
${\bf V}=(V_1,\ldots, V_k)$ of commuting isometries on a Hilbert space $\cG\supset \cH$ such that

\begin{equation*}
 F({\bf X})= \sum_{m_1\in \ZZ}\cdots \sum_{m_k\in \ZZ} ({ V}_{1}^*)^{m_1^-}\cdots ({V}_{k}^*)^{m_k^-}{V}_{1}^{m_1^+}\cdots ({V}_{k})^{m_k^+}
\otimes { X}_{1}^{m_1^-}\cdots {X}_{k}^{m_k^-}({X}_{1}^*)^{m_1^+}\cdots ({X}_{k}^*)^{m_k^+}
\end{equation*}
for any ${\bf X}=(X_1,\ldots, X_k)\in {\bf D}^k(\cH)$,
where the convergence is in the operator norm topology.
Due to It\^o's theorem (see \cite{SzFBK-book}), there is a $k$-tuple    ${\bf U}=(U_1,\ldots, U_k)$ is of  commuting unitaries on a Hilbert space $\cK\supset \cG$ such that $U_i|_\cG=V_i$ for any $i\in \{1,\ldots, k\}$. Due to Fuglede's theorem (see \cite{Dou}), the unitaries are doubly commuting, i.e. $U_iU_j^*=U_j^*U_i$ for any $i,j\in \{1,\ldots, k\}$. Consequently, we have $
F({\bf X})=(P_\cE\otimes I)\boldsymbol{\cP}({\bf U}, {\bf X})|_{\cE\otimes \cH}.
$
The converse of the theorem is due to Theorem 4.2 from \cite{Po-pluriharmonic-polyball}. The proof is complete.
\end{proof}

Given a completely  bounded linear  map $\mu:\text{\rm span}\{\boldsymbol\cR_{\bf n}^*\boldsymbol\cR_{\bf n}\}\to B(\cE)$, we  introduce   the {\it noncommutative Poisson transform} of $\mu$ to
be the map \ $\boldsymbol\cP\mu : {\bf B}_{\bf n}(\cH)\to B(\cE)\otimes_{min}B(\cH)$
defined by
$$
(\boldsymbol\cP \mu)({\bf X}):=\widehat \mu[\boldsymbol{\cP}({\bf R}, {\bf X})],\qquad
{\bf X} \in {\bf B}_{\bf n}(\cH),
$$
 where  the  completely bounded linear map
 $$\widehat
\mu:=\mu\otimes \text{\rm id} : \text{\rm span}\{\boldsymbol\cR_{\bf n}^*\boldsymbol\cR_{\bf n}\}^{-\|\cdot \|} \otimes_{min} B(\cH)\to
B(\cE)\otimes_{min} B(\cH)
$$
is uniquely defined by  $ \widehat \mu(A\otimes Y):= \mu(A)\otimes Y$ for any $A\in
\text{\rm span}\{\boldsymbol\cR_{\bf n}^*\boldsymbol\cR_{\bf n}\} $ and  $Y\in B(\cH)$.

  Using Corollary 4.5 from \cite{Po-pluriharmonic-polyball} and Proposition \ref{polydisk},  we obtain the following   structure theorem for  positive  $k$-harmonic functions on the regular polydisk  ${\bf D}^k(\cH)$,
    which extends the corresponding  classical result in scalar polydisks \cite{Ru1}.

\begin{theorem}\label{cp} Let $F:{\bf D}^k(\cH)\to B(\cE)\otimes_{min} B(\cH)$ be  a free $k$-pluriharmonic function. Then the following statements are equivalent:
\begin{enumerate}
\item[(i)]
  $F$ is positive;
\item[(ii)]
there exists a completely positive linear map $\mu:C^*({\bf R})\to B(\cE)$ such that  $F=\boldsymbol\cP\mu$;

\item[(iii)]
there exists a
$k$-tuple   ${\bf U}=(U_1,\ldots, U_k)$  of commuting unitaries acting on  a Hilbert space $\cK\supset \cE$  and a bounded operator $W:\cE\to \cK$ such  that
$$
F({\bf X})=  (W^*\otimes I)\left[ C_{\bf X}({\bf U})^*
C_{\bf X}({\bf U}) \right](W\otimes I),
$$
\end{enumerate}
where $$C_{\bf X}({\bf U}):=(I\otimes \boldsymbol\Delta_{\bf X}(I)^{1/2})\prod_{i=1}^k(I-U_{i}\otimes
X_{i}^*), \qquad {\bf X}=(X_1,\ldots, X_k)\in {\bf D}^k(\cH).$$
 \end{theorem}

\begin{proof} In the particular case when $F(0)=I$, the implication (i)$\implies $(ii) is due to Corollary 4.5 from \cite{Po-pluriharmonic-polyball}.
Now, we consider the general case when $F:{\bf D}^k(\cH)\to B(\cE)\otimes_{min} B(\cH)$ is  an arbitrary positive  free $k$-pluriharmonic function
of the form
\begin{equation*}
  \sum_{m_1\in \ZZ}\cdots \sum_{m_k\in \ZZ} A_{{m_1^-}\cdots {m_k^-};{m_1^+}\cdots {m_k^+}}\otimes { X}_{1}^{m_1^-}\cdots {X}_{k}^{m_k^-}
({X}_{1}^*)^{m_1^+}\cdots ({X}_{k}^*)^{m_k^+}
\end{equation*}
for any ${\bf X}=(X_1,\ldots, X_k)\in {\bf D}^k(\cH)$,
where the convergence is in the operator norm topology.
For each $\epsilon >0$, set
$$
G_\epsilon:=\left[(A+\epsilon I_\cE)^{-1/2}\otimes I\right](F+\epsilon I_\cE\otimes I) \left[(A+\epsilon I_\cE)^{-1/2}\otimes I\right],
$$
where $A\otimes I:=F(0)$ with $A\geq 0$.
Since $G_\epsilon$ is a positive free $k$-pluriharmonic function on ${\bf D}^k(\cH)$ with $G_\epsilon(0)=I$, we can apply Corollary 4.5 from
\cite{Po-pluriharmonic-polyball} and a find a completely positive linear map $\mu_\epsilon:C^*({\bf R})\to B(\cE)$ such that
$$
\mu_\epsilon\left(({\bf R}_{1}^*)^{m_1^-}\cdots ({\bf R}_{k}^*)^{m_k^-}{\bf R}_{1}^{m_1^+}\cdots {\bf R}_{k}^{m_k^+}\right)=
(A+\epsilon I_\cE)^{-1/2} A_{{m_1^-}\cdots {m_k^-};{m_1^+}\cdots {m_k^+}} (A+\epsilon I_\cE)^{-1/2}.
$$
Define the completely positive linear map $\nu_\epsilon:C^*({\bf R})\to B(\cE)$ by setting
$$
\nu_\epsilon (g):=(A+\epsilon I_\cE)^{1/2}  \mu_\epsilon(g)(A+\epsilon I_\cE)^{1/2}, \qquad g\in C^*({\bf R}).
$$
Note that $
\nu_\epsilon\left(({\bf R}_{1}^*)^{m_1^-}\cdots ({\bf R}_{k}^*)^{m_k^-}{\bf R}_{1}^{m_1^+}\cdots {\bf R}_{k}^{m_k^+}\right)=
 A_{{m_1^-}\cdots {m_k^-};{m_1^+}\cdots {m_k^+}}
$ if $(m_1,\ldots, m_k)\neq (0,\ldots, 0)$, and $\nu_\epsilon(I)=A+\epsilon I_\cE$.
Define $\mu:C^*({\bf R})\to B(\cE)$ by setting
$$
\mu\left(({\bf R}_{1}^*)^{m_1^-}\cdots ({\bf R}_{k}^*)^{m_k^-}{\bf R}_{1}^{m_1^+}\cdots {\bf R}_{k}^{m_k^+}\right)=
 A_{{m_1^-}\cdots {m_k^-};{m_1^+}\cdots {m_k^+}}
$$
 if $(m_1,\ldots, m_k)\neq (0,\ldots, 0)$, and $\mu(I)=A$.
It is clear that $\nu_\epsilon (g)=\mu(g)+\epsilon \left<g(1),1\right> I$ for any $g\in C^*({\bf R})$, and $\nu_\epsilon (g)\to \mu(g)$ as $\epsilon\to 0$.
Therefore, $\mu$ is a completely  positive linear map and
$F=\boldsymbol\cP\mu$.
Using  Corollary 4.5 from \cite{Po-pluriharmonic-polyball} and Proposition \ref{polydisk}, one deduce the implications (ii)$\implies$(iii) and (iii)$\implies$ (i). The proof is complete.
\end{proof}

We recall \cite{JP} that the Kobayashi distance for the polydisc $\DD^k$ is given by
$$
K_{\DD^k}({\bf z}, {\bf w})=\frac{1}{2}\ln \frac{1+\|\psi_{\bf z}({\bf w})\|_\infty}{1-\|\psi_{\bf z}({\bf w})\|_\infty},
$$
 where   $\psi_{\bf z}$  is the involutive automorphisms of $\DD^k$  given by
 $$
 \psi_{\bf z}=\left(\frac{w_1-z_1}{1-\bar z_1 w_1}, \ldots, \frac{w_k-z_k}{1-\bar z_k w_k}\right)
 $$
for any ${\bf z}=(z_1,\ldots, z_k)$ and ${\bf w}=(w_1,\ldots, w_k)$   in $\DD^k$.  
\begin{theorem}\label{disk}
 Let ${\bf D}^k(\cH)$ be the regular polydisk. The following statements hold.
\begin{enumerate}
\item[(i)]  If ${\bf A}, {\bf B}\in {\bf D}^k(\cH)^-$, then ${\bf A}\,\overset{H}{\sim}\, {\bf B}$ if and only if ${\bf A}\,\overset{P}{\sim}\, {\bf B}$.
\item[(ii)]
The metrics $\delta_H$ and $\delta_\cP$ coincide on the Harnack parts of ${\bf D}^k(\cH)^-$.
\item[(iii)] If ${\bf A}$ and ${\bf B}$ are   in
 ${\bf D}^k(\cH)^-$ and ${\bf A}\,\overset{H}{\sim}\, {\bf B}$, then
$$
\delta_H({\bf A}, {\bf B})=\delta_H(\boldsymbol\Psi({\bf A}), \boldsymbol\Psi({\bf B})), \qquad \boldsymbol\Psi\in
Aut({\bf D}^k).
$$
\item[(iv)] If ${\bf A}$ and ${\bf B}$ are   in
 ${\bf D}^k(\cH)$, then
$$
\delta_H({\bf A}, {\bf B})=\ln \max \left\{\left\|C_{\bf A}({\bf R})C_{\bf B}({\bf R})^{-1}\right\|, \left\|C_{\bf B}({\bf R})C_{\bf A}({\bf R})^{-1}\right\|\right\},
$$
where
$$C_{\bf X}({\bf R}):=(I\otimes \boldsymbol\Delta_{\bf X}(I)^{1/2})\prod_{i=1}^k(I-R_{i}\otimes
X_{i}^*), \qquad {\bf X}=(X_1,\ldots, X_k)\in {\bf D}^k(\cH). $$
\item[(v)]
$\delta_H|_{\DD^k\times \DD^k}$ is equivalent to the  Kobayashi distance on the polydisk $\DD^k$ and
$$\delta_H({\bf z}, {\bf w})=\frac{1}{2} \ln \frac{\prod_{i=1}^k\left(1+|\psi_{z_i}(w_i)|\right)}
{\prod_{i=1}^k\left(1-|\psi_{z_i}(w_i)|\right)}
$$
for any ${\bf z}=(z_1,\ldots, z_k)$ and ${\bf w}=(w_1,\ldots, w_k)$ in $\DD^k$, where $\psi_{\bf z}:=(\psi_{z_1},\ldots, \psi_{z_n})$ is the involutive automorphisms of $\DD^k$ such that
$\psi_{z_i}(0)=z_i$ and $\psi_{z_i}(z_i)=0$.
\item[(vi)] The hyperbolic metric $\delta_H$ is complete on the Harnack parts of ${\bf D}^k(\cH)_0^-$.
\item[(vii)]
The $\delta_H$-topology coincides with the operator norm topology on the regular polydisk ${\bf D}^k(\cH)$.
\end{enumerate}
\end{theorem}
\begin{proof} Let ${\bf A}, {\bf B}\in {\bf D}^k(\cH)^-$ and recall that the map  ${\bf X}\mapsto \boldsymbol{\cP}({\bf R}, {\bf X})$ is a positive free $k$-pluriharmonic function on ${\bf D}^k(\cH)$. Consequently, if
${\bf A}\overset{H}{{\underset{c}\prec}}\, {\bf B}$, then ${\bf A}\overset{P}{{\underset{c}\prec}}\, {\bf B}$. To prove the converse, assume that ${\bf A}\overset{P}{{\underset{c}\prec}}\, {\bf B}$. Then we have
\begin{equation}
\label{pcp}
\boldsymbol{\cP}({\bf R}, r{\bf A})\leq c^2 \boldsymbol{\cP}({\bf R}, r{\bf B})
\end{equation}
 for any $r\in [0,1)$. Let $F:{\bf D}^k(\cH)\to B(\cE)\otimes_{min} B(\cH)$ be  an arbitrary positive  free $k$-pluriharmonic function with coefficients in $B(\cE)$. According to Theorem \ref{cp}, there exists a completely positive linear map $\mu:C^*({\bf R})\to B(\cE)$ such that  $F({\bf X})=(\boldsymbol\cP\mu)({\bf X})=(\mu\otimes id)[\boldsymbol{\cP}({\bf R}, {\bf X})]$ for any ${\bf X}\in {\bf D}^k(\cH)$. Consequently, using
 relation \eqref{pcp}, we deduce that $F(r{\bf A})\leq c^2F(r{\bf B})$
 for any $r\in [0,1)$. This shows that ${\bf A}\overset{H}{{\underset{c}\prec}}\, {\bf B}$ and  completes the proof of item (i). As a consequence, we deduce that
 the Harnack parts of ${\bf D}^k(\cH)^-$ coincide with the Poisson parts. Moreover, since
${\bf A}\overset{H}{{\underset{c}\sim}}\, {\bf B}$ if and only if ${\bf A}\overset{P}{{\underset{c}\sim}}\, {\bf B}$, part (ii) holds.
Note that part (iii)  is a particular case of Theorem \ref{formula}, while part (iv) is a particular case of Theorem \ref{dp}. On the other hand, the items (vi) and (vii) are due to part (i) and Theorem \ref{dp-dh}.

It remains to  prove part (v). Due to part (iv), we have
\begin{equation}\label{deh}
\delta_H({\bf z},0)=\ln \max \left\{ \|C_{\bf z}({\bf R})\|, \|C_{\bf z}({\bf R})^{-1}\|\right\},
\end{equation}
where  $$C_{\bf z}({\bf R})=\prod_{i=1}^k (1-|z_i|^2)^{1/2} \prod_{i=1}^k (I-\bar z_i R_i)^{-1}.
$$
Since $\left\|(I-\bar z_i R_i)^{-1}\right\|\leq \frac{1}{1-|z_i|}$, we deduce that
\begin{equation}\label{Cz}
\|C_{\bf z}({\bf R})\|\leq \prod_{i=1}^k\left(\frac{1+|z_i|}{1-|z_i|}\right)^{1/2}.
\end{equation}
Now, we calculate $ \|C_{\bf z}({\bf R})^{-1}\|$. First, note that
$$
\left\|\prod_{i=1}^k(I-\bar z_i R_i)\right\|\leq \prod_{i=1}^k (1+|z_i|).
$$
Due to Riesz representation theorem we have
$$\
\sup_{(w_1,\ldots, w_k)\in \DD^k}\left|\prod_{i=1}^k (1+\bar z_iw_i)\right|=
\prod_{i=1}^k(1+ |z_i| ).
$$
The von Neumann inequality for regular polyballs (see \cite{Po-Berezin-poly}) implies
$$
\left|\prod_{i=1}^k (1+\bar z_iw_i)\right|\leq \left\|\prod_{i=1}^k(I-\bar z_i R_i)\right\|.
$$
Now, combining these relations, we deduce that
$$
\left\|\prod_{i=1}^k(I-\bar z_i R_i)\right\|=\prod_{i=1}^k (1+|z_i|),
$$
which implies
$$
 \|C_{\bf z}({\bf R})^{-1}\|=\prod_{i=1}^k\left(\frac{1+|z_i|}{1-|z_i|}\right)^{1/2}.
$$
Hence and using relations \eqref{deh} and \eqref{Cz}, we obtain
\begin{equation}
\label{dhz}
\delta_H({\bf z},0)=\frac{1}{2} \ln \prod_{i=1}^k\frac{1+|z_i|}{1-|z_i|}.
\end{equation}
Now, let $\psi_{\bf z}:=(\psi_{z_1},\ldots, \psi_{z_n})$  be the involutive automorphism of $\DD^k$ such that
$\psi_{z_i}(0)=z_i$ and $\psi_{z_i}(z_i)=0$.
Using part (iii) and relation \eqref{dhz}, we deduce that
\begin{equation*}
\begin{split}
\delta_H({\bf z}, {\bf w})&=\delta_H(\psi_{\bf z}({\bf z}), \psi_{\bf z}({\bf w}))
 =\delta_H(0, \psi_{\bf z}({\bf w}))\\
 &=\frac{1}{2} \ln \frac{\prod_{i=1}^k\left(1+|\psi_{z_i}(w_i)|\right)}
{\prod_{i=1}^k\left(1-|\psi_{z_i}(w_i)|\right)}=\sum_{i=1}^k \delta_{\DD}(z_i,w_i),
\end{split}
\end{equation*}
where $\delta_{\DD}$ is the Poincar\' e distance on the open disk $\DD$.
Since the function  $t\mapsto \ln \frac{1+t}{1-t}$ is increasing on $[0,1)$, we have
$$
\max_{i=1,\ldots, k} \frac{1}{2} \ln \frac{1+|\psi_{z_i}(w_i)|}{1-|\psi_{z_i}(w_i)|}
=\frac{1}{2} \ln \frac{1+\|\psi_{\bf z}({\bf w})\|_\infty}{1-|\psi_{\bf z}({\bf w})\|_\infty},
$$
which is the Kobayashi distance for the polydisc (see \cite{Ko1}). Consequently, $\delta_H|_{\DD^k\times \DD^k}$ is equivalent to the  Kobayashi distance on the polydisk $\DD^k$.  This completes the proof of part
(v).
\end{proof}

\begin{corollary}
Let $f=(f_1,\ldots, f_m):{\bf D}^k(\cH)\to   [B(\cH)^m]_1$  be a free holomorphic function on the regular polydisk.
If ${\bf X}, {\bf Y}\in {\bf D}^k(\cH)$,
then
$$
\delta_H(f({\bf X}), f({\bf Y}))\leq \delta_H({\bf X}, {\bf Y}),
$$
where $\delta_H$ is the hyperbolic  metric.    In particular, if $f(0)=0$, then
$$
\frac{1+\|f({\bf z})\|_2}{1-\|f({\bf z})\|_2}\leq \prod_{i=1}^k \frac{1+|z_i|}{1-|z_i|}
$$
for any ${\bf z}=(z_1,\ldots, z_k)$  in $\DD^k$.
\end{corollary}
\begin{proof} The first part of this corollary is due to Proposition \ref{SP1}. To prove the second part, we use Theorem \ref{disk} part (ii) and (v).
Note that
$$\delta_H({\bf z}, 0)=\frac{1}{2}\ln \prod_{i=1}^k \frac{1+|z_i|}{1-|z_i|}\quad
\text{ and } \quad
\delta_H(f({\bf z}),0)= \frac{1}{2}\ln \frac{1+\|f({\bf z})\|_2}{1-\|f({\bf z})\|_2}.
$$
Since $\delta_H(f({\bf z}),0)\leq \delta_H({\bf z}, 0)$,  one can complete the proof.
\end{proof}

\bigskip

       %


\begin{thebibliography}{99}






\bibitem{AST} {\sc T.~Ando, I.~Suciu, D.~ Timotin},
 Characterization of some Harnack parts of contractions,
{\it J. Operator Theory} {\bf  2} (1979), no. 2, 233--245.



\bibitem{Arv-acta} {\sc W.B.~Arveson}, Subalgebras of $C^*$-algebras,
{\it Acta Math.} {\bf 123} (1969), 141--224.







\bibitem{Be} {\sc S.~ Bergman},
 {\it The kernel function and conformal mapping},
 Mathematical Surveys, No. V. American Mathematical Society,
Providence, R.I., 1970. x+257 pp.






\bibitem{Dou} {\sc   R. G.~Douglas}, {\em Banach algebra techniques in operator theory}, Second edition. Graduate Texts in Mathematics, {\bf 179}, Springer-Verlag, New York, 1998. xvi+194 pp.


\bibitem{Fo} {\sc C.~ Foia\c s},  On Harnack parts of contractions.
 {\it Rev. Roumaine Math. Pures Appl.} {\bf  19} (1974), 315--318.







\bibitem{Ha1} {\sc L.A.~Harris},
 Bounded symmetric homogeneous domains in infinite dimensional spaces,
 {\it Proceedings on Infinite Dimensional Holomorphy (Internat. Conf.,
 Univ. Kentucky, Lexington, Ky.}, 1973),
 pp. 13--40. {\it Lecture Notes in Math.}, Vol. 364, Springer, Berlin, 1974.


\bibitem{Ha2} {\sc L.A.~Harris},
Schwarz-Pick systems of pseudometrics for domains in normed linear
spaces, {\it Advances in holomorphy (Proc. Sem. Univ. Fed. Rio de
Janeiro, Rio de Janeiro}, 1977), pp. 345--406, North-Holland Math.
Stud., {\bf 34}, North-Holland, Amsterdam-New York, 1979.



\bibitem{Ha3} {\sc L.A.~Harris},
Analytic invariants and the Schwarz-Pick inequality, {\it Israel J.
Math.} {\bf 34} (1979), no. 3, 177--197.

\bibitem{JP} {\sc M.~Jarnicki, P.~Pflug}, {\it Invariant distances and metrics in complex analysis}, vol. 9, Walter de Gruyter, 1993.



\bibitem{KV}
 {\sc D. S.~Kaliuzhnyi-Verbovetskyi and V.~Vinnikov},
    {\it Foundations of free noncommutative function theory}, Mathematical Surveys and Monographs, {\bf 199}, American Mathematical Society, Providence, RI, 2014. vi+183 pp.

\bibitem{Ko1} {\sc S.~Kobayashi}, {\it Hyperbolic complex spaces}, Grundlehren der Mathematischen Wissenschaften [Fundamental Principles of Mathematical Sciences], 318.
 Springer-Verlag, Berlin, 1998. xiv+471 pp.

\bibitem{Ko2} {\sc S.~Kobayashi}, {\it  Hyperbolic manifolds and holomorphic mappings},
 World Scientific Publishing Co. Pte. Ltd., Hackensack, NJ, 2005. xii+148 pp.





\bibitem{Kr2} {\sc S.G.~Krantz}, {\em Function theory of several complex variables}. Reprint of the 1992 edition. AMS Chelsea Publishing, Providence, RI, 2001. xvi+564 pp.


\bibitem{Kr1} {\sc S.G.~Krantz}, {\it  Geometric function theory},
 Explorations in complex analysis.
 Cornerstones. Birkhäuser Boston, Inc., Boston, MA, 2006. xiv+314 pp.




\bibitem{Pa-book} {\sc V.I.~Paulsen},
 {\it Completely Bounded Maps and Dilations},
Pitman Research Notes in Mathematics, Vol.146, New York, 1986.








\bibitem{Pi-book} {\sc G.~Pisier}, \emph{ Similarity Problems and Completely Bounded Maps},
Springer Lect. Notes Math., Vol.1618, Springer-Verlag, New York,
1995.











      \bibitem{Po-poisson} {\sc G.~Popescu},
     {Poisson transforms on some $C^*$-algebras generated by isometries,}
       {\it J. Funct. Anal.} {\bf 161} (1999),  27--61.





      \bibitem{Po-holomorphic} {\sc G.~Popescu},
      {Free holomorphic functions on the unit ball of $B(\cH)^n$},
      {\it J. Funct. Anal.}  {\bf 241} (2006), 268--333.








\bibitem{Po-pluriharmonic} {\sc G.~Popescu},
{Noncommutative transforms and free pluriharmonic functions},
 {\it Adv. Math.} {\bf 220} (2009), 831--893.


\bibitem{Po-hyperbolic3} {\sc G.~Popescu}, Hyperbolic geometry on the unit ball of $B(H)^n$   and dilation theory,  {\it Indiana Univ. Math. J.}  {\bf 57}  (2008),  no. 6, 2891--2930.

 \bibitem{Po-hyperbolic} {\sc G.~Popescu},
      {Noncommutative hyperbolic  geometry on the unit ball of $B(\cH)^n$},
      {\it J. Funct. Anal.}  {\bf 256} (2009), 4030--4070.


\bibitem{Po-hyperbolic2} {\sc G.~Popescu}, {Hyperbolic geometry on noncommutative balls}, {\it  Doc. Math.}  {\bf 14}  (2009), 595--651.




\bibitem{Po-automorphism} {\sc G.~Popescu},
{ Free holomorphic automorphisms of the unit ball of $B(\cH)^n$},
{\it J. Reine Angew. Math.} {\bf 638} (2010), 119--168.




\bibitem{Po-Berezin3}  {\sc G.~Popescu}, {Berezin transforms on
noncommutative  varieties in polydomains}, {\it  J. Funct. Anal.} {\bf 265} (2013), no. 10, 2500--2552.




\bibitem{Po-curvature-polyball} {\sc G.~Popescu}, {Curvature invariant  on noncommutative polyballs},  {\it Adv. Math} {\bf 279} (2015), 104--158.

\bibitem{Po-Berezin-poly} {\sc G.~Popescu}, {Berezin transforms on noncommutative polydomains},   {\it Trans. Amer. Math. Soc.} {\bf 368} (2016), 4357--4416.

\bibitem{Po-Euler-charact} {\sc G.~Popescu}, {Euler characteristic on noncommutative polyballs}, {\it J. Reine Angew. Math.}, to appear.

\bibitem{Po-automorphisms-polyball} {\sc G.~Popescu}, {Holomorphic automorphisms of noncommutative polyballs}, {\it J. Operator Theory}  {\bf 76} (2016), no.2, 387--448.

\bibitem{Po-pluriharmonic-polyball} {\sc G.~Popescu},
{Free pluriharmonic functions on noncommutative polyballs}, {\it Analysis \& PDE} {\bf 9} (2016), 1185--1234.





\bibitem{Ru1} {\sc W.~Rudin},
{\em Function theory in polydiscs}, W. A. Benjamin, Inc., New York-Amsterdam 1969 vii+188 pp.


\bibitem{Ru2} {\sc W.~Rudin},
{\em Function theory in the unit ball of \,$\CC^n$}, {
Springer-verlag, New-York/Berlin}, 1980.




\bibitem{Su2} {\sc I.~Suciu}, Analytic formulas  for the hyperbolic
distance between two contractions, {\it  Ann. Polon. Math.} {\bf 66}
(1997), 239--252.


 \bibitem{Su3} {\sc I.~Suciu},
 The Kobayashi distance between two contractions.
 {\it Operator extensions, interpolation of functions and related topics }(Timisoara, 1992),
  189--200,
   Oper. Theory Adv. Appl.  {\bf 61}, Birkhäuser, Basel, 1993.




















%



\bibitem{SzFBK-book} {\sc B.~Sz.-Nagy, C.~Foia\c{s}, H.~Bercovici, and L.~K\' erchy}, {\em Harmonic
Analysis of Operators on Hilbert Space}, Second edition. Revised and enlarged edition. Universitext. Springer, New York, 2010. xiv+474 pp.


\bibitem{Up} {\sc H.~Upmeier}, {\em Symmetric Banach manifolds and
Jordan $C^*$-algebras}, North-Holland Mathematics Studies {\bf 104},
North Holland, 1985.



\bibitem{Zhu} {\sc K.~Zhu}, {\it Spaces of holomorphic functions in the unit
ball},  Graduate Texts in Mathematics, 226. Springer-Verlag, New
York, 2005. x+271 pp.






       \end{thebibliography}
      \end{document}